\documentclass[rm]{birkmult}

\usepackage{amsmath}
\usepackage{amssymb}
\usepackage{amsfonts}

\usepackage{amsthm}
\usepackage{graphics}
\usepackage[square]{natbib}

\input cyracc.def

\def \kbar {\overline{k}}

\def \Aa {{\mathcal A}}
\def \Cc {{\mathcal C}}
\def \Ff {{\mathcal F}}
\def \Gg {{\mathcal G}}
\def \Hh {{\mathcal H}}
\def \Jj {{\mathcal J}}
\def \Kk {{\mathcal K}}
\def \Ll {{\mathcal L}}

\def \Pp {{\mathcal P}}

\def \Ss {{\mathcal S}}
\def \Ss {{\mathcal S}}
\def \Ww {{\mathcal W}}

\def \SSS {{\mathfrak S}}

\def \P {\mathbb P}

\def \Z {\mathbb Z}

\def \O {\mathfrak O}
\def \XXX {\mathfrak X}
\def \UUU {\mathfrak U}
\def \BBB {\mathfrak B}
\def \CCC {\mathfrak C}
\def \DDD {\mathfrak D}
\def \a {\mathfrak a}

\def \f {\mathfrak f}
\def \h {\mathfrak h}

\def \p {\mathfrak p}
\def \r {\mathfrak r}
\def \u {\mathfrak u}

\def \int {{\rm int}}

\def \Inv {{\rm Inv}}

\def \GL {{\rm GL}}

\def \Pic {{\rm Pic}\,}

\def \Sel {{\rm Sel}}

\def \kbar {\bar{k}}

\def\harr#1#2{\smash{\mathop{\hbox to .5in{\rightarrowfill}}
\limits^{\scriptstyle#1}_{\scriptstyle#2}}}

\def\varr#1#2{\llap{$\scriptstyle #1$}\left\downarrow \vcenter to
.5in{}\right.\rlap{$\scriptstyle #2$}}

\newtheorem{theorem}{Theorem}[section]
\newtheorem{corollary}[theorem]{Corollary}
\newtheorem{lemma}[theorem]{Lemma}
\newtheorem{proposition}[theorem]{Proposition}

\theoremstyle{definition}
\newtheorem{definition}[theorem]{Definition}

\theoremstyle{remark}
\newtheorem{remark}[theorem]{Remark}
\newtheorem{notation}[theorem]{Notation}

\numberwithin{equation}{section}

\begin{document}
%\received{April 4, 2008}

%\revised{November 3, 2008}

%\accepted{November 10, 2008}
\title[Dual Kummer Surface]{Explicit Computations on the Desingularized Kummer Surface}

\author{V.G. Lopez Neumann}

\address{Faculdade de Matem\'atica,
          Universidade Federal de Uberl\^andia,
          MG - Brazil}
\email{gonzalo@famat.ufu.br}

\author{Constantin Manoil}
\address{Section de Math\'ematiques,
          Universit\'e de Gen\`eve,
          CP $64$, $1211$ Geneva $4$, Switzerland}
\email{constantin.manoil@math.unige.ch}

\begin{abstract}
We find formulas for the birational maps from a Kummer surface
$\Kk$ and its dual $\Kk^*$ to their common minimal
desingularization $\Ss$. We show how the nodes of $\Kk$ blow up.
Then we give a description of the group of  linear automorphisms
of $\Ss$.
\end{abstract}

%\thanks{The first author was partially supported by FAPEMIG}


%\subjclass{Primary 11G30, 14L30; Secondary 11G10, 14L17}


\date{}

\maketitle

%\begin{document}
\hsize=13.5truecm
%\title{Explicit Computations on the Desingularized Kummer Surface}
%\author{Victor Gonzalo Lopez Neumann \& Constantin Manoil}

\maketitle

%\tableofcontents
%\parindent=0pt

\section{Introduction} The Kummer surface is a mathematical object
having a long history and which has been considered from various
points of view. We present the two approaches of this topic that
are relevant for this paper.

In the 19-th century a singular surface $\Kk$, called the Kummer
surface, was attached to a quadratic line complex. A minimal
desingularization $\Sigma$ of $\Kk$  and a birational map
$\kappa_1:\Kk\dashrightarrow \Sigma$ were  constructed by
geometric methods. One may call this the "old", or {\sl classical}
construction of the Kummer, which we recall in Section $4$.

Another construction is the following: let $A$ be an abelian
surface, and let the involution $\sigma: A\longrightarrow A$ be
given by $\sigma(x)=-x$. The quotient $\Kk=A/\sigma$ has $16$
double points and one defines a K3 surface $\Ss$ to be $\Kk$ with
these $16$ nodes blown up (see \cite{Be}, Prop. 8.11). The
presently prevailing terminology in literature designates $\Ss$ as
the Kummer surface attached to $A$. However  in this paper, to be
consistent with the historical point of view and with our main
reference \cite{CF}, we call $\Ss$ the {\sl desingularized} Kummer
surface. If $\Kk=\Jj(\Cc)/\sigma$, where $\Jj(\Cc)$ is the
Jacobian of a curve $\Cc$ of genus $2$, then $\Kk$ is called the
Kummer surface {\sl belonging to} $\Cc$. The natural question of
the connection between the two constructions of the Kummer surface
was explicitly answered in \cite{CF}, Chapter $17$: given the
equations of $\Cc$, one can write down the equations of the
quadratic line complex which will yield, by classical
construction, the Kummer surface belonging to $\Cc$ (see Lemma
\ref{lameme}).

The Jacobian $\Jj(\Cc)$ can be embedded in $\P^{15}$ and is
described by $72$ quadratic equations (\cite{Fl}). This makes
explicit calculations with the Jacobian difficult, for instance
giving equations of twists. Methods have been looked for, to study
the Mordell-Weil group of $\Jj(\Cc)$ using more computable
objects.

A desingularization $\Ss$ of $\Kk$ is constructed explicitly in
\cite{CF}, Chapter $16$ by algebraic methods. The surface $\Ss$
appeared naturally in recent attempts to compute the Mordell-Weil
group, by using important tools as Cassels' morphism and the fake
Selmer group. The K3 surface $\Ss$ is a smooth intersection of
three quadrics in $\P^5$. Denote by $\Kk^*$ the projective dual of
$\Kk$.  There are birational maps
$$
\kappa : \Kk \dashrightarrow \Ss \quad \text{and}\quad \kappa^* :
\Kk^* \dashrightarrow \Ss.
$$
There are morphisms extending $\kappa^{-1} : \Ss\dashrightarrow
\Kk$ and ${\kappa^*}^{-1} : \Ss\dashrightarrow \Kk^*$ to all of
$\Ss$, which we also denote by $\kappa^{-1}$ and ${\kappa^*}^{-1}$
and these are  minimal desingularizations of $\Kk$ and $\Kk^*$.

{\bf Origins}. \;Cassels and Flynn explain that  the surface $\Ss$
comes from the behaviour of six of the tropes (see Definition
\ref{trop}) under the duplication map. The existence of $\Ss$
raises more far-reaching questions. Indeed, if the ground field
$k$ is algebraically closed, one always has a commutative diagram:
\begin{equation}\label{dual}
\begin{array}{ccc}
\Jj(\Cc)                 & \harr{d_0}{} &  \Jj(\Cc)^0  \\
\varr{\pi}{} &              & \varr{}{\pi_0} \\
\Kk\,                  & \harr{d_0^*}{} & \Kk^*
\end{array}
\end{equation}
where $\Jj(\Cc)^0$ is the dual of $\Jj(\Cc)$ as an abelian
variety. Here, the maps $d_0$ and $d_0^*$ depend on the choice of
a rational point on $\Cc$. In other words, the abelian varieties
duality matches with the projective one (see \cite{CF}, Note $1$
on page $35$). When $k$ is not algebraically closed, one has to
enlarge the ground field to obtain such diagrams, yet $\Ss$ is a
desingularization {\sl over $k$} of both $\Kk$ and $\Kk^*$. One
can ask if there is  a unifying object for $\Jj(\Cc)$ and
$\Jj(\Cc)^0$, generalizing the abelian varieties duality.

{\bf Recent developments}. \;Perhaps the most natural definition
of $\Ss$ and its twists is related to an idea of M. Stoll and N.
Bruin, in connection with the computation of the Mordell-Weil
group of $\Jj(\Cc)$. This was presented a few years after
\cite{CF} appeared. We give a brief account of it in Section $5$.

Cassels and Flynn already suggested that the $2$-Selmer group
could be investigated by using twists of $\Ss$. In $2007$ A. Logan
and R. van Luijk (\cite{LL}) and P. Corn (\cite{C}) made use of
twists of $\Ss$ to find specific curves with nontrivial
$2$-torsion elements in the Tate-Shafarevich groups of their
Jacobians.

{\bf Our results} and structure of this paper.  In Section $2$ we
give a   background.

In Sections $3$ and $4$ we achieve part of the program suggested
in \cite{CF} at the end of Chapter $16$. In Section $3$ we
complete the construction in \cite {CF}, enlarge the set of points
where $\kappa:\Kk\dashrightarrow \Ss$ is explicitly defined and
find formulae for $\kappa$, listed in the Appendix. These formulae
allow one to describe how each singularity of $\Kk$ blows up,
beyond the simple use of the Kummer structure. Our treatment is
algebraic, in the vein of the definition of $\Ss$.

In Section $4$ we turn to geometric language. We first recall the
classical construction of $\Kk$. In Proposition \ref{conect} we
give an explicit isomorphism $\Theta$ between $\Sigma$ and $\Ss$.
Then we observe that well-known projective isomorphisms $W_i$
between $\Kk$ and $\Kk^*$ lift to $\Sigma$ to correspondences
given by line complexes of degree $1$. Using $\Theta$ we show that
the $W_i$ lift to $\Ss$ to known commuting involutions (Corollary
\ref{epsilon2}). With the formulae for $\kappa$ at hand one can
now find formulae for $\kappa^*$.

In Section $5$ we give a description of the group of  linear
automorphisms of $\Ss$.

This paper is almost self-contained. Moreover, all essential
information needed is contained in \cite{CF}.

\section{Preliminaries}

We work over a field $k$ of characteristic different from $2$. By
a curve, we mean a smooth projective irreducible variety of
dimension~$1$. {\sl Throughout this paper, $\Cc$ will be a curve
of genus~$2$}. If $k$ has at least $6$ elements and char($k)\neq
2$, such a curve admits an affine model:

\begin{equation}\label{courbe}
\Cc' : \quad Y^2 = F(X),
\end{equation}

where
$$
F(X)=f_0+f_1X+\ldots +f_6X^6 \in k[X], \qquad f_6\neq0
$$
and $F$ has distinct roots~$\theta_1, \ldots , \theta_6$. The
points $\a_i=(\theta_i , 0)$ are the Weierstrass points. We denote
by $\infty^\pm$ the points at infinity on the completion $\Cc$ of
$\Cc'$. For a given point~$\r=(x,y)$ on~$\Cc$, the {\sl conjugate
of~$\r$ under the~$\pm Y$ involution} is the point ${\bar \r} =
(x,-y)$. In accordance, for a divisor~$\XXX = \sum n_i \r_i$ we
denote by~${\overline \XXX} = \sum n_i {\bar \r_i}$. The class of
a divisor $\XXX$ is denoted by $[\XXX]$ and a divisor in the
canonical class  by $K_{\Cc}$.

There is a bijection between\, $\Pic^0(\Cc)$ and \,$\Pic^2(\Cc)$
defined by \,$[\DDD]\mapsto [\DDD+K_{\Cc}]$. Hence one can regard
a point of the Jacobian $J(\Cc)$ as the class of a divisor~$\XXX
=\r + \u$, where~$\r = ( x , y )$,~$\u = (u,v)$ is a pair of
points on~$\Cc$.

%%%%%%%%%%%%%%%%%%%%%%%%%%%%%%%%%%%%%
%%%%%%%%%%%%%%%%%%%%%%%%%%%%%%%%%%%%%
%%%%%%%%%%%%%%%%%%%%%%%%%%%%%%%%%%%%%
%%%%%%%%%%%%%%%%%%%%%%%%%%%%%%%%%%%%%
%%%%%%%%%%%%%%%%%%%%%%%%%%%%%%%%%%%%%

Effective divisors of degree~$2$ can be identified with points on
the symmetric product~$\Cc^{(2)}$ of~$\Cc$ with itself. This is a
non-singular variety, since $\Cc$ is a non-singular curve. Now,
the canonical class is represented by an infinity of divisors
\,$\r+{\bar \r}$, while any other class in $\Pic^2(\Cc)$ has a
unique representative. Hence the Jacobian should look like
$\Cc^{(2)}$ with the representatives of~[$K_{\Cc}]$ "blown down".
Flynn finds a projective emdedding of the Jacobian (see~\cite{Fl})
in the following way.

For a point~$\XXX=\{ \r , \u \}$ on~$\Cc^{(2)}$ with~$\r=(x,y)$,
$\u=(u,v)$, define:
$$
\sigma_0 =1, \quad \sigma_1 = x+u, \quad \sigma_2 =xu,
$$

$$
\beta_0= \frac{F_0(x,u) - 2yv}{(x-u)^2}, \quad
$$
where
$$
\begin{array}{ll}
F_0 (x,u) = & 2f_0+f_1(x+u)+2f_2xu+f_3xu(x+u) \cr
            & +2f_4(xu)^2+f_5(xu)^2(x+u)+2f_6(xu)^3 .
\end{array}
$$

The Jacobian is then the projective locus of ${\bf z} =(z_0:
\ldots :z_{15})$ in $\P^{15}$, where~$z_0=\delta$; $z_1=\gamma_1$;
$z_2=\gamma_0$; $z_i = \beta_{5-i}$,~$i=3,4,5$;
$z_i=\alpha_{9-i}$, $i=6,\ldots ,9$; $z_i =\sigma_{14-i}$,
$i=10,\ldots ,14$; and $z_{15}=\rho$. For the definition of the
functions~$\alpha$, $\beta$, etc. and details, see~\cite{Fl}.
\vskip 0.2truecm\noindent
%%%%%%%%%%%%%%%%%%%%%%%%%%%%%%%%%%%%%%%%%%%%%%%%
Using the bijection between\, $\Pic^0(\Cc)$ and \,$\Pic^2(\Cc)$
one can describe generically the group law $\oplus$ on $\Jj(\Cc)$,
with neutral element $[K_{\Cc}]$. Let \,$\UUU$, $\BBB$ be divisors
such that \,$\UUU+\BBB\ne \r+{\bar \r}+\CCC$, with
\,$\r\in\Cc(\kbar)$ and $\CCC$ divisor of degree $2$. Then there
is a unique  $M(X)\in \kbar[X]$ of degree $3$ such that the cubic
\,$Y=M(X)$ passes through the four points of $\UUU$, $\BBB$. The
complete intersection of the cubic curve with $\Cc$ is given by
$$
M(X)^2=F(X), \qquad Y=M(X).
$$
The residual intersection is an effective divisor $\DDD$. Then
\,$[\UUU]\oplus [\BBB]=[\overline{\DDD}]$, i.e. \,$[\UUU]\oplus
[\BBB]=[(x_5,-M(x_5))+(x_6,-M(x_6))]$, where $x_5, x_6$ are the
last two roots of $M(X)^2-F(X)$.

\begin{definition}
The Kummer surface~$\Kk$ belonging to a curve of genus~$2$, is the
projective locus in~$\P^3$ of the elements\;
$\xi=(\xi_1:\xi_2:\xi_3:\xi_4)$,\; where
\begin{equation}\label{Kummer}
\xi_1= \sigma_0 , \quad \xi_2= \sigma_1 , \quad \xi_3= \sigma_2 ,
\quad \xi_4= \beta_0 .
\end{equation}
\end{definition}
The equation of the Kummer surface is given in~\cite{CF},
formula~($3.1.9$). It is of the form
\begin{equation}\label{KummerEq}
\Kk :\quad K = K_2 \xi^2_4 + K_1 \xi_4 + K_0 = 0
\end{equation}
where the~$K_i$ are forms of degree $4 - i$ in $\xi_1, \xi_2,
\xi_3$. The natural map from~$J(\Cc)$ to~$\Kk$ given by
$$
(z_0:\ldots : z_{15}) \longmapsto
(z_{14}:z_{13}:z_{12}:z_5)=(\xi_1:\ldots : \xi_4)
$$
is $2$ to $1$; the ramification points correspond to divisor
classes~$[\XXX]$ with~$[ \XXX]=[ {\overline \XXX}]$. In the sense
of the group law of the Jacobian, this means
that~$2[\XXX]=[K_{\Cc}]$. The images of these classes are the $16$
singular points (nodes) on~$\Kk$: $N_0 = (0:0:0:1)$ corresponding
to $[K_{\Cc}]$ and other~$\left(
\begin{array}{c} 6 \cr 2 \end{array} \right) = 15$ nodes~$N_{ij}$
corresponding to classes of divisors~$\XXX_{ij}=\a_i + \a_j$ with
$i \neq j$.
\begin{definition}\label{even}
A function $f\in\kbar(\Jj(\Cc))$ is called {\sl even} if $f([{\bar
\r}])=f([\r])$ and {\sl odd} if $f([{\bar \r}])=-f([\r])$.
\end{definition}
\begin{definition}\label{Kdual}
The surface $\Kk^*\subset(\P^3)^\checkmark=\P^3$ is the projective
dual of $\Kk$, i.e. to a point $\xi\in\Kk$ corresponds the point
$\eta\in\Kk^*$ such that
$\eta=(\eta_1:\eta_2:\eta_3:\eta_4)\in(\P^3)^\checkmark$ gives the
tangent plane to $\Kk$ at $\xi$.
\end{definition}
\begin{definition}\label{trop}
There are $6$ planes $T_i$ containing the $6$ nodes $N_0$ and
$N_{ij}$, $j \neq i$ and $10$ planes $T_{ijk}$ containing the $6$
nodes $N_{mn}$ for $\{ m,n \} \subset \{i,j,k \}$ or $\{ m,n \}
\cap \{i,j,k \} = \emptyset$. These are the {\it tropes}; they cut
conics on $\Kk$. They correspond to the $16$ singular points of
$\Kk^*$.
\end{definition}

When using the term trope, it will be clear from the context if we
refer to planes, conics or singular points of $\Kk^*$. The
equations of the $T_i$ are given in (\ref{tro}).
\section{The Desingularized Kummer}\label{sectiondesingularized}
We recall the facts from \cite{CF} Chapter $16$ we need, keeping
the notation there. We start with a  simple, yet a bit technical,
explanation of the ideas leading to the construction of the
desingularized Kummer $\Ss$. For a more conceptual one, see
\cite{CF}, Chapter $6$, Section $6$.

Recall that the Kummer parametrizes divisor of degree $2$, modulo
linear  equivalence and $\pm Y$ involution. Let $[\XXX] = [(x,y) +
(u,v)]\neq[K_{\Cc}]$, where $yv \neq 0$ and $x \neq u$ be a
divisor class. There is a unique $M(X)$ of degree $3$ such that
\begin{equation}\label{twiceM1}
M(X)^2-F(X)=(X-x)^2(X-u)^2 H(X),
\end{equation}
for a quadratic $H(X)$. The divisor given by \,$H(X)=0$,
\,$Y=-M(X)$ is in the class $2[\XXX]$. There is a unique
polynomial $P^*(X)$ of degree at most $5$ such that
\begin{equation}\label{polasterisque1}
(X-x)(X-u)P^*(X) \equiv M(X) \text{ mod } F(X).
\end{equation}
Then
$$
(X-x)^2(X-u)^2 {P^*}^2(X) \equiv M^2(X) \equiv (X-x)^2(X-u)^2 H(X)
\text{ mod } F(X)
$$
and since~$F(x)F(u)=(yv)^2 \neq 0$, we have also
$$
{P^*}^2(X) \equiv H(X) \text{ mod } F(X) .
$$
Changing $\XXX$ to $\overline{\XXX}$, changes $M(X)$ to $-M(X)$,
so  $P^*(X)$ to $-P^*(X)$.

Conversely, given $P^*$ with $\deg(P^*)\leq 5$ and
$(P^*)^2\equiv\text{quadratic \,mod } F$, the equation
\begin{equation}
(X-x)(X-u)P^*(X) \equiv \text{cubic} = M(X) \text{ mod } F(X)
\end{equation}
puts $2$ conditions on $x$, $u$, so has in general a unique set of
solutions. The divisor classes of $\DDD=(x, M(x))+(u, M(u))$ and
$\overline{\DDD}$ give the same point on the  Kummer and
correspond to $P^*$ and $-P^*$. All this suggests the following
construction.

%Let~$\Kk$ be the Kummer surface belonging to the genus~$2$ curve
%defined by (\ref{courbe}).
Let~${\bf p}=(p_0:\ldots : p_5)$, where the~$p_j$ are
indeterminates, and put~$P(X)=\sum_0^5 p_j X^j$. Let~$\Ss$ the
projective locus of the~${\bf p}$ for which~$P(X)^2$ is congruent
to a quadratic in~$X$ modulo~$F(X)$. Put
\begin{equation}\label{Pj}
P_j(X)= \prod_{i \neq j} (X- \theta_i)
\end{equation}
and $\omega_j=P_j(\theta_j) \neq 0$. Since $\theta_i\ne\theta_j$
for $i\ne j$, we have $\omega_j\ne 0$ and the~$P_j$ span the
vector space of polynomials of degree at most~$5$. We have
\begin{equation}\label{linearcombination}
P(X)= \sum_j \pi_j P_j (X) , \qquad\text{where} \qquad
\pi_j=\frac{P(\theta_j)}{\omega_j} .
\end{equation}

%Note that the~$\pi_j$ are linear forms in $p_j$.
%One verifies that
%$$
%P(X)^2 \equiv \sum \omega_j \pi_j^2 P_j (X) \text{ mod } F(X) .
%$$
The~$K3$ surface~$\Ss$ given as the complete intersection in
$\P^5$ of the three quadrics $\Ss_0$, $\Ss_1$, $\Ss_2$ where
\begin{equation}\label{p20}
\Ss_i : S_i=0, \qquad\text{and}\qquad S_i=\sum_j \theta_j^i
\omega_j \pi_j^2\quad \text{for} \quad i=0,1,2
\end{equation}
%\begin{equation}\label{DesKummer}
%\displaystyle
%\begin{aligned}
%S_0 & = \sum_j \omega_j \pi_j^2 \\
%S_1 & = \sum_j \theta_j \omega_j \pi_j^2 \\
%S_2 & = \sum_j \theta_j^2 \omega_j \pi_j^2
%\end{aligned}
%\end{equation}
is a minimal desingularization of $\Kk$ and also of $\Kk^*$. Here
the $S_i$ are quadratic forms in ${\bf p}$ with coefficients in
$\Z[f_1,\ldots ,f_6]$.

%%%%%%%%%%%%%%%%%%%%%%%%%%%%%%%%%%%%%%%%%%%%%%%%%%%
%%%%%%%%%%%%%%%%%%%%%%%%%%%%%%%%%%%%%%%%%%%%%%%%%%%
%%%%%%%%%%%%%%%%%%%%%%%%%%%%%%%%%%%%%%%%%%%%%%%%%%%
%%%%%%%%%%%%%%%%%%%%%%%%%%%%%%%%%%%%%%%%%%%%%%%%%%%
%%%%%%%%%%%%%%%%%%%%%%%%%%%%%%%%%%%%%%%%%%%%%%%%%%%

The following theorems hold (\cite{CF}, Theorems 16.5.1 and
16.5.3):

\begin{theorem}\label{theoA}
There is a birational map $\kappa : \Kk \dashrightarrow \Ss$
defined for general $\xi \in \Kk$ as follows:

Let $\XXX = \{ (x,y) , (u,v) \}$ correspond to~$\xi$. Put
$G(X)=(X-x)(X-u)$ and let $M(X)$ be the cubic determined by the
property that $Y-M(X)$ vanishes twice on $\XXX$. Let
$P(X)=\sum_0^5 p_j X^j$ be determined by $GP \equiv M \text{ mod }
F$. Then $\kappa(\xi)$ is the point with projective coordinates
$(p_0:\ldots : p_5)$.
\end{theorem}

Let~$\kappa^* : \Kk^* \dashrightarrow \Ss$ be the birational map
defined in \cite{CF}, Theorem $16.5.2$.

\begin{theorem}\label{theoB}
Let~$\xi \in \Kk$ and~$\eta \in \Kk^*$ be dual, that is~$\eta$
gives the tangent to~$\Kk$ at~$\xi$.
Then~$\kappa(\xi)=\kappa^*(\eta)$.
\end{theorem}

Our first result is the following.
\begin{lemma}\label{mapkappa}
The map~$\kappa : \Kk\dashrightarrow\Ss$ from Theorem~\ref{theoA}
is given by the formulae listed in the Appendix.
\end{lemma}

%%%%%%%%%%%%%%%%%%%%%%%%%%%%%%%%%%%%%%%%%%%%%%%%%%%
%%%%%%%%%%%%%%%%%%%%%%%%%%%%%%%%%%%%%%%%%%%%%%%%%%%
%%%%%%%%%%%%    Preuve   %%%%%%%%%%%%%%%%%%%%%%%%%%
%%%%%%%%%%%%%%%%%%%%%%%%%%%%%%%%%%%%%%%%%%%%%%%%%%%
%%%%%%%%%%%%%%%%%%%%%%%%%%%%%%%%%%%%%%%%%%%%%%%%%%%
\begin{proof}
The problem is to make effective the method given in \cite{CF},
Chapter $16$. For completeness and because of typing errors there,
we recall it in the Appendix. As presented in \cite{CF}, the
method works for a general element $[\XXX] = [(x,y) + (u,v)]$,
where $y v \neq 0$ and $x \neq u$.
%There is a unique $M(X)$ of
%degree $3$ such that
%\begin{equation}\label{twiceM}
%M(X)^2-F(X)=(X-x)^2(X-u)^2 H(X)
%\end{equation}
%for a quadratic $H(X)$. There is a unique polynomial $P^*(X)$ of
%degree at most $5$ such that
%\begin{equation}\label{polasterisque}
%(X-x)(X-u)P^*(X) \equiv M(X) \text{ mod } F(X).
%\end{equation}
%Then
%$$
%(X-x)^2(X-u)^2 {P^*}^2(X) \equiv M^2(X) \equiv (X-x)^2(X-u)^2 H(X)
%\text{ mod } F(X),
%$$
%and since~$F(x)F(u)=(yv)^2 \neq 0$, we have also
%$$
%{P^*}^2(X) \equiv H(X) \text{ mod } F(X) .
%$$
After finding the polynomial $P$ from Theorem \ref{theoA}, one has
to modify it slightly, so that its coefficients be even functions
of~$k(J(\Cc))$ and therefore belong to the function field of the
Kummer.

We get formulae for~$\kappa$ by expressing~$y^2=F(x)$, $v^2=F(u)$,
$$
yv=\frac{F_0(x,u)-\beta_0 (x-u)^2}{2}
$$
in the coefficients of~$P(X)$, then as the resulting coefficients
are symmetric fonctions of $x$ and~$u$, we express  them in terms
of $\xi_2=x+u$ and~$\xi_3=xu$. Finally we homogenize the formulae
with respect to~$\xi_1=1$, $\xi_2$, $\xi_3$, $\xi_4=\beta_0$.

%The coefficients of $P^\star(X)$ are now homogeneous polynomials of
%degree $6$ on $\xi$. As it is suggested by the formula
%(\ref{polcombination}), we may write $p_i=\alpha_j - \beta_j
%\xi_4^2$. We obtain the formulae:

One first obtains
$$
\kappa(\xi)=
 \left(
        \tilde{p}_0(\xi): \ldots : \tilde{p}_5(\xi)
 \right) ,
$$
where
\begin{equation}\label{polcombination}
\tilde{p}_j(\xi)= \alpha_j  K_2 + \beta_j  (K_1 \xi_4 + K_0),
\quad \text{for } 0 \le j \le 5.
\end{equation}
Here $\alpha_j$ and $\beta_j$ are homogeneous forms in $\xi$ of
degree $4$ and $2$ respectively and the $K_j$ are those in
(\ref{KummerEq}).

Taking $p_j(\xi)=(\tilde{p}_j(\xi) - \beta_j K)/K_2 = \alpha_j -
\beta_j \xi_4^2$, we obtain formulae of degree $4$ for $\kappa$
which will be defined also for $K_2=0$, extending $\kappa$ to
images of divisor classes $[\XXX] = [(x,y) + (u,v)]$ with $x=u$
and $y=v \neq 0$. However the formulae do not work for  points
with $F'(x)=0$ and for the image of $[2 \, \infty^+]$. We will
treat the case $y=0$ or $v=0$ in connection with nodes and tropes.
\end{proof}
%%%%%%%%%%%%%%%%%%%%%%%%%%%%%%%%%%%%%%%%%%%%%%%%%%%
%%%%%%%%%%%%%%%%%%%%%%%%%%%%%%%%%%%%%%%%%%%%%%%%%%%
%%%%%%%%%%%%   end Preuve   %%%%%%%%%%%%%%%%%%%%%%%
%%%%%%%%%%%%%%%%%%%%%%%%%%%%%%%%%%%%%%%%%%%%%%%%%%%
%%%%%%%%%%%%%%%%%%%%%%%%%%%%%%%%%%%%%%%%%%%%%%%%%%%

%In the equations (\ref{DesKummer}) of~$\Ss$, the~$\pi_j$ occur only
%as squares. Over the algebraic closure there is a unique $\tilde{P}
%(X)$ such that
%\begin{equation}\label{automorphism}
%\tilde{P} (\theta_j)= \varepsilon_j P(\theta_j)
%\end{equation}
%for any choice of~$\varepsilon_j= \pm 1$. We
%denote~$\varepsilon^{(i)}$ the automorphism~(\ref{automorphism})
%with the choice~$\varepsilon_i=-1$ and~$\varepsilon_j=1$, for $j
%\neq i$.
There are 6 commuting involutions $\varepsilon^{(i)}$ of $\Ss$,
which can be described as follows. Let the point $(p_0: \ldots :
p_5)$ be represented by the polynomial $P(X)=p_0+p_1X+\cdots
p_5X^5$ and
% we can make
%these involutions $\varepsilon^{(i)}$ explicit in the following way:
let
\begin{equation}\label{Gj}
 g_i(X) = 1 - 2 \frac{P_i(X)}{P_i(\theta_i)},
 \end{equation}
where $P_i(X)$ is defined by (\ref{Pj}). We see that
$g_i(\theta_j)=(-1)^{\delta_{ij}}$, so $g_i(X)^2 \equiv 1 \mod
F(X)$. Then one defines
\begin{equation}\label{formulavarepsilon}
\varepsilon^{(i)}(P(X)) = g_i(X) P(X) \mod F(X).
\end{equation}
In terms of coordinates $\pi_j$, one has
\begin{equation}\label{piji}
\varepsilon^{(i)}(\pi_j)=(-1)^{\delta_{ij}}\pi_j.
\end{equation}
\begin{definition}\label{Inv}
We define $\text{Inv} (\Ss)$ to be the group of  $32$ commuting
involutions of $S$ generated by the $\varepsilon^{(i)}$.
\end{definition}

The~${\bf p}=(p_0:p_1:0:0:0:0)$ are clearly in~$\Ss$ and form a
rational line~$\Delta_0$. We shall often write
$p_0+p_1X\in\Delta_0$. Acting on $\Delta_0$ by the involutions
gives~$31$ further lines.

\begin{notation}\label{Delta}
We denote:
$$
\Delta_i= \varepsilon^{(i)} (\Delta_0) \, , \quad
\Delta_{ij}=\varepsilon^{(i)} \circ \varepsilon^{(j)} (\Delta_0)
\quad \text{and} \quad \Delta_{ijk}=\varepsilon^{(i)} \circ
\varepsilon^{(j)} \circ \varepsilon^{(k)}(\Delta_0) \, .
$$
\end{notation}

%%%%%%%%%%%%%%%%%%%%%%%%%%%%%%%%%%%%%%%%%%%%%%%%%%%
%%%%%%%%%%%%%%%%%%%%%%%%%%%%%%%%%%%%%%%%%%%%%%%%%%%
%%%%%%%%%%%%%%%   begin lemma  %%%%%%%%%%%%%%%%%%%%
%%%%%%%%%%%%%%%%%%%%%%%%%%%%%%%%%%%%%%%%%%%%%%%%%%%
%%%%%%%%%%%%%%%%%%%%%%%%%%%%%%%%%%%%%%%%%%%%%%%%%%%
We come now to the main result in this section, which describes
how the singularities of $\Kk$ and $\Kk^*$ blow up.

\begin{lemma}\label{structure}
The map $\kappa$ blows up the node~$N_0=(0:0:0:1)$ of~$\Kk$ into
the line~$\Delta_0$ and the $15$ nodes $N_{ij}$ into the lines
$\Delta_{ij}$. The tropes $T_i$ and $T_{ijk}$ blow up by
$\kappa^*$ into the lines $\Delta_i$ and $\Delta_{ijk}$.
\end{lemma}
{\sl Note.} This result is predicted in \cite{CF} as plausible.
\begin{proof}
The node~$N_0$ corresponds to the canonical class, therefore we
consider divisors of the type~$\XXX=(x,y)+(u,v)$ with~$u=x+h$, $h$
small and~$v \approx -y \neq 0$. Then the local behaviour of the
Kummer coordinates is $\xi_1=1$, $\xi_2=2x+h \approx 2x$,
$\xi_3=x(x+h) \approx x^2$ and
$$
\xi_4 =\frac{F_0(x,x+h) - 2yv}{h^2} \approx \frac{4y^2}{h^2}.
$$
Replacing this in the formulae for~$\kappa$ and clearing
denominators, then taking the limit as~$h \to 0$, we obtain
$$
\begin{array}{lcl}
\kappa (\xi) & \approx & (-16 x y^4 : 16 y^4 : 0 : 0 : 0 : 0 ) \cr
             & \approx & ( -x : 1 : 0 : 0 : 0 : 0 ) ,
\end{array}
$$
since~$y \neq 0$.

Note that for $(X-\theta_i) \in \Delta_0$, we have
$$
\varepsilon^{(i)}(X-\theta_i) \equiv g_i(X) (X-\theta_i) \equiv (X
- \theta_i) \mod F(X),
$$
so $\Delta_0$ and $\Delta_i$ intersect at $(-\theta_i: 1 :
0:0:0:0)$.

We now show that $\Delta_0 \cap \Delta_{ij} = \emptyset$ for $i
\neq j$. Indeed, the intersection point $\p$ should be invariant
by $\varepsilon^{(i)} \circ \varepsilon^{(j)}$. A polynomial
$P(X)$ represents such a point iff
$$
\begin{array}{c}
\alpha P(X) \equiv g_i(X) g_j(X)P(X) \mod F(X) \quad \text{for
some
} \, \alpha \in \kbar^*  \\
\text{iff} \\
F(X) \mid P(X)(\alpha - g_i(X)g_j(X)) .
\end{array}
$$
Replacing $X$ by the roots of $F(X)$ one sees that $P(X)$ must
have at least two roots among the $\theta_k$, so it must be of
degree at least $2$ and therefore cannot represent a point on
$\Delta_0$.
%for they meet at a point
%$$
%\varepsilon^{(i)} (p_0,p_1,0,0,0,0)= (\tilde{p}_0,\tilde{p}_1,0,0,0,0),
%$$
%where
%$$
%\begin{array}{lll}
%{\tilde p}_0 + {\tilde p}_1 \theta_j & = &
%          p_0+p_1 \theta_j  \qquad \text{for } j \neq i \cr
%{\tilde p}_0 + {\tilde p}_1 \theta_i & = &
%          - (p_0 + p_1 \theta_i),
%\end{array}
%$$
%by~(\ref{automorphism}).
%
%Since~$\theta_k \neq \theta_j$ ($k \neq j$) we see that~${\tilde
%p}_0=p_0$ and~${\tilde p}_1=p_1$ and the last relation is possible
%iff $p_0 + p_1 \theta_i =0$, hence~$\Delta_0 \cap \Delta_i =
%(-\theta_i, 1 , 0,0,0,0)$. Again, since~$\theta_i \neq \theta_j$ a
%similar argument shows that~$\Delta_0 \cap \Delta_{ij} =
%\emptyset$  for $i \neq j$. Applying~$\varepsilon^{(i)}$ to this
%and since~${\varepsilon^{(i)}}^2=1$ we have~$\Delta_i \cap
%\Delta_j = \emptyset$ for~$i \neq j$.
Similarly, $\Delta_0 \cap \Delta_{ijk} = \emptyset$ for $i \neq j
\neq k$.

The six $\Delta_i$ are strict transforms of the conics cut
on~$\Kk$ by the tropes containing~$N_0$.
%By theorem~\ref{theoB}
%they are also blow-ups of the corresponding nodes of~$\Kk^*$. We
%can find parametric equations for them.
To see this, recall that we still have to define~$\kappa$ for
points corresponding to divisors~$\XXX= \{x,y\} + \{ \theta_i , 0
\}$ with $y \neq 0$. Write $F(X)=f_6(X-\theta_i) P_i (X)$. From
this we get formulae for $f_k$, $k=0, \ldots ,6$ depending
on~$\theta_i$ and~$h_{ij}$, $j=0, \ldots 5$, the coefficients
of~$P_i(X)$, which we plug into
$$
\xi_4=\frac{F_0(x,\theta_i)}{(x-\theta_i)^2} .
$$
We substitute then~$\xi_1=1$, $\xi_2=x+\theta_i$,
$\xi_3=x\theta_i$ and~$\xi_4$ in the formulae for~$\kappa$. On
multiplying by $(x-\theta_i)^2/(f_6^2 P_i(x))$ (note that
$P_i(x)\neq 0$), we obtain
\begin{equation}\label{eqdeltai}
P(X)= 2(x-\theta_i)P_i(X) + P_i(\theta_i)(X-x) ,
\end{equation}
that is
\begin{equation}\label{eq2.7}
\begin{array}{lll}
 p_0 & = & 2h_{i0} (x-\theta_i) -P_i(\theta_i)x  \cr
 p_1 & = & 2h_{i1} (x-\theta_i) +P_i(\theta_i)  \cr
 p_j & = & 2h_{ij} (x-\theta_i) \qquad \qquad
\text{for } 2 \le j \le 5 .
\end{array}
\end{equation}
The points~$(1: x+\theta_i: x\theta_i: \xi_4)$  belong to the
conic $T_i$ cut on~$\Kk$ by the trope
\begin{equation}\label{tro}
\theta_i^2 \xi_1 - \theta_i \xi_2 + \xi_3 = 0,
\end{equation}
passing throught~$N_0$ and~$N_{ij}$, $j\neq i$. Formulae
$(\ref{eq2.7})$ give parametric equations (in $x$) of the strict
transform by $\kappa$ of this conic. To confirm that this is
$\Delta_i$, one verifies that
$$
P(X) \equiv P_i(\theta_i)g_i(X)(X-x) \mod F(X) .
$$
Recall from~\cite{CF}, Chapter~$4$, Section $5$ that there are
linear maps~$W_i : \Kk\longrightarrow \Kk^*$, taking the
node~$N_0$ to the trope~$T_i$ ($W_i$ is induced by addition of a
Weierstrass point $\a_i$). Further, $W_i^{-1} \circ W_j$ moves
$N_0$ to the node~$N_{ij}$. Applying the results in Section $4$
and especially Corollary \ref{epsilon2}, one concludes that:

1)\, the tropes $T_i$ considered as  singular points of $\Kk^*$,
blow up by $\kappa^*$ into $\Delta_i$;

2)\, each of the fifteen $N_{ij}$ blows up into $\Delta_{ij}$;

3)\, the tropes $T_{ijk},\, i\neq j\neq k$ blow up into
$\Delta_{ijk}$;
%($i \neq j$) meets the lines $\Delta_i$,
%~$\Delta_j$, and~$\Delta_{ijk}$, with~$k\notin\{i,j\}$, so
%the~$\Delta_{ij}$ are blow-ups of the fifteen nodes~$N_{ij}$,
%since $N_0$ and $N_{ij}$ are the only nodes both on $T_i$ and
%$T_j$.

4)\, the ten $\Delta_{ijk}$ ($i\neq j \neq k$) are strict
transforms of the ten conics cut on $\Kk$ by the tropes (planes)
not containing~$N_0$. Each of them intersects six $\Delta_{ij}$
since each node is on six tropes.
\end{proof}
\section{Linear and quadratic line complexes}

We recall from \cite{Hu} and~\cite{CF} the definitions of the
Kummer surface and of the corresponding desingularization $\Sigma$
in terms of quadratic complexes. Then we link them to the surface
$S$ studied in Section $3$. Notations are like in \cite{CF},
Chapter~$17$.

If~$u=(u_1:u_2:u_3:u_4)$ and~$v=(v_1:v_2:v_3:v_4)$ are distinct
points in~$\P^3$, then the Grassman coordinates of the line~$<u,v>
\subset \P^3$ are
$$
\mathfrak{p}=(p_{43}:p_{24}:p_{41}:p_{21}:p_{31}:p_{32})=(X_1:\ldots
: X_6),
$$
with~$p_{ij}=u_iv_j-u_jv_i$. Denote by~$\Gg$ the Grassmanian
quadric in~$\P^5$, representing the lines in~$\P^3$. Its equation
is
$$
G(X_1,\ldots ,X_6)=2X_1X_4+2X_2X_5+2X_3X_6=0.
$$
\begin{definition}\label{linear}
A line complex of degree $d$ is a set of lines in~$\P^3$ whose
Grassman coordinates satisfy a homogeneous equation $Q(X_1,\ldots
, X_6)=0$ of degree $d$.
\end{definition}
If $d=1$ this is called a linear complex and if $d=2$ a quadratic
complex.

A line $L \in \Gg$ parametrizes a pencil of lines in $\P^3$. The
lines of a pencil $L$ all pass through a point $\f(L)=u$, called
the {\it focus} of the pencil) and lie in one plane $\h(L)=\pi_u$,
the {\it plane} of the pencil.

All lines in a linear complex~$\Ll$ passing through a given
point~$u$ (respectively lying in a plane $\pi$), form a pencil
$L_u$ (respectively $L_\pi$). Each linear complex $\Ll$
establishes a {\it correspondence} between points and planes in
$\P^3$:
$$
I(u)=\h(L_u), \quad I(\pi)=\f(L_\pi), \quad I^2=1,
$$
which is defined also for lines; if~$l \subset \P^3$ is the line
$<u,u'>$, then~$I(l)=I(u) \cap I(u')$. The line~$I(l)$ is the {\it
polar} line of~$l$ with respect to the given linear complex.
\begin{definition}
Two linear complexes are called apolar if the correspondences they
define commute.
\end{definition}

Let~$H$ be any  quadratic form in six variables such  that the
quadrics~$G=0$ and~$H=0$ intersect transversely and denote
by~$\Hh=\{x\in\P^5 \mid H(x)=0\}$. Let~$\Ww=\Gg\cap\Hh$ and
~$\Aa=$
 set of lines on~$\Ww$. The points in~$\Ww$ represent the lines
in~$\P^3$ whose Grassman coordinates~$\mathfrak{p}$ satisfy
$H(\mathfrak{p})=0$. A line~$L\in\Aa$ represents a pencil of lines
of this quadratic complex in~$\P^3$.
\begin{definition}
The Kummer surface~$\Kk \subset \P^3$ associated to the quadratic
complex~$\Hh$ is the locus of focuses of such pencils: \,
$\Kk=\{\f(L) \mid L\in \Aa\}$.
\end{definition}

\begin{definition}
The dual Kummer surface~$\Kk^* \subset {\P^3}^\checkmark$
associated to the quadratic complex~$\Hh$ is the locus of planes
of such pencils.
\end{definition}

From now on we suppose $f_6=1$.
\begin{lemma}\label{lameme}
For any curve~$\Cc$ of genus~$2$, the Kummer surface belonging to
the curve~$\Cc$ given by (\ref{courbe}) coincides with the Kummer
surface just defined, if one takes the quadratic complex $\Hh$ to
be given by
$$
\begin{array} {ll}\label{Kummer2}
H = &-4X_1X_5-4X_2X_6-X_3^2+2f_5X_3X_6+4f_0X_4^2\\
    &+4f_1X_4X_5+4f_2X_5^2 +4f_3X_5X_6+(4f_4-f_5^2)X_6^2.
\end{array}
$$
\end{lemma}
\begin{proof}
See \cite{CF}, Lemma~$17.3.1$ and pages $182 - 183$.
\end{proof}
Now, if a point~$\xi \in \P^3$ is the focus of the pencil
corresponding to the line~$L_\xi \in \Aa$, then~$L_\xi$ lies in
the plane~$\Pi_\xi \subset \Gg$ corresponding to lines in~$\P^3$
passing through~$\xi$. But then the conic~$\Pi_\xi \cap \Hh$
contains~$L_\xi$, so is degenerate; $\Pi_\xi$ is tangent to $\Hh$
and~$\Pi_\xi \cap \Hh = L_\xi \cup L'_\xi$. The lines of the
quadratic complex passing through~$\xi$ are in the two pencils
$L_\xi$ and~$L'_\xi$, each with focus~$\xi$, lying in the
planes~$\pi_\xi$ and~$\pi'_\xi$ in $\P^3$. The point $\xi$ is a
{\it singular point} of the quadratic complex. The line
$l_\xi=\pi_\xi \cap \pi'_\xi$ is represented on $\Gg$ by the
point~$\mathfrak{p}_\xi=L_\xi \cap L'_\xi$ and is called a {\it
singular line} of the quadratic complex.

If~$L_\xi \neq L'_\xi$ the pencils are distinct and $\xi$ is a
simple point of the Kummer; there is a one-to-one correspondence
$\xi \leftrightarrow \mathfrak{p}_\xi$. However, if
$L_\xi=L'_\xi$, then~$\pi_\xi=\pi'_\xi$ and all the lines
in~$L_\xi$ are singular lines. The point~$\xi$ is a singular point
of the Kummer, because the map $\mathfrak{f} : \Aa\longrightarrow
\Kk$ is algebraic. Therefore the variety $\Sigma$ parametrizing
singular lines is a desingularization of the Kummer.
\begin{remark}
The quadrics $\Ss_0$ and $\Ss_2$ in the defining equation
(\ref{p20}) of $\Ss$  are dual with respect to $\Ss_1$ and Cassels
and Flynn call for an interpretation of this duality. More
precisely, if $Q\in\Ss_0$ (respectively $Q\in\Ss_2$) then the
hyperplane
$$
\sum_{j=0}^5{\frac{\partial{S_1}}{\partial{\pi_j}}(Q)\cdot\pi_j}=0
$$
is tangent to $\Ss_2$, respectively to $\Ss_0$. Now, it is shown
in~\cite{Hu}, Section~$31$ that, if one brings~$G$ and~$H$ to
diagonal form :
$$
G : \sum_{i=1}^{6} {X_i}^2=0 \qquad \text{and}\qquad H :
\sum_{i=1}^{6} \alpha_i{X_i}^2=0,
$$
then the variety parametrizing the singular lines is
$$
\Sigma=\Gg \cap \Hh \cap \Ff \qquad\text{where}\qquad \Ff :
\sum_{i=1}^{6} \alpha^2_i{X_i}^2=0
$$
(see also~\cite{CF}, Corollary~$2$ to Lemma~$17.2.1$). Over
$\kbar$ one can change~$\pi_j\leftrightarrow \sqrt{\omega_j}\pi_j$
and so the equations of~$\Ss_0$ and~$\Ss_1$ give that of~$\Ss_2$.
This explains the duality of~$\Ss_0$ and~$\Ss_2$ with respect
to~$\Ss_1$.
\end{remark}
\begin{definition}\label{kappa1}
The birational map $\kappa_1 : \Kk \dashrightarrow \Sigma$ is
defined by $\kappa_1(\xi)=\p_\xi$.
\end{definition}
\begin{definition}\label{kappa1*}
The birational map $\kappa^*_1 : \Kk^* \dashrightarrow \Sigma$
associates to a plane $\pi$ tangent to~$\Kk$ the intersection
point of the lines in $\Aa$ parametrizing the two pencils in $\Hh$
contained in~$\pi$.
\end{definition}
It is shown in \cite{CF} and \cite{Hu} that  $\kappa_1^{-1}$ and
${\kappa_1^*}^{-1}$ extend to minimal desingularizations
$\kappa_1^{-1}:\Sigma\longrightarrow \Kk$  and
${\kappa_1^*}^{-1}:\Sigma\longrightarrow \Kk^*$.

\begin{lemma}\label{duality}
The surface $\Kk^*$ is the projective dual of $\Kk$ that is, if
$\xi=\f(L) \in \Kk$ then $\eta = \h(L) \in \Kk^*$ is the tangent
plane of $\Kk$ at $\xi$. Therefore
$\kappa_1(\xi)=\kappa^*_1(\eta)$.
\end{lemma}
\begin{proof}
See \cite{CF}, page $181$.
\end{proof}
\subsection{Connection between $\Ss$ and $\Sigma$}
Denote by $G(\overrightarrow{X}, \overrightarrow{Y})$ the bilinear
form associated to the Grassmanian $G$. Make the change of
coordinates
\begin{equation}\label{XZE}
\zeta_i =
\frac{G(\overrightarrow{X},\overrightarrow{v}(\theta_i))}{\sqrt{\omega_i}},
\end{equation}
with vectors $\overrightarrow{v}(\theta_i)$ as in \cite{CF}
formula ($17.4.3$). The desingularisation of the Kummer surface
corresponding to the quadratic complex $H$ of Lemma \ref{lameme}
is the~$K3$ surface~$\Sigma$ given as the complete intersection in
$\P^5$ of the three quadrics $\Sigma_0$, $\Sigma_1$, $\Sigma_2$
where
$$
\Sigma_i :\qquad
 \sum_j \theta_j^i \zeta_j^2 = 0 \quad \text{for}
\quad i=0,1,2.
$$
%\begin{equation}\label{DesKummer}
%\displaystyle
%\begin{aligned}
%\Sigma_0 & : \sum_j \zeta_j^2 = 0\\
%\Sigma_1 & : \sum_j \theta_j \zeta_j^2 = 0 \\
%\Sigma_2 & : \sum_j \theta_j^2 \zeta_j^2 =0
%\end{aligned}
%\end{equation}
(see also \cite{Hu}, Section~$31$).

Let $\Theta : \Sigma \longrightarrow \Ss$ be defined by
\begin{equation}\label{THETA}
\Theta(\zeta_1:\dots :\zeta_6) =
 \left(
 \frac{\zeta_1}{\sqrt{\omega_1}}:\dots:\frac{\zeta_6}{\sqrt{\omega_6}}
 \right) =
 (\pi_1:\dots:\pi_6) .
\end{equation}
To write $\Theta$ in variables  $X_j$ on $\Sigma$ and $p_j$ on
$\Ss$, recall that  \,$P(X)=\sum_0^5 p_j X^j$ and note that by
(\ref{linearcombination}) and (\ref{XZE}):
$$
\frac{P(\theta_i)}{\omega_i}=\pi_i =
 \frac{\zeta_i}{\sqrt{\omega_i}} =
  \frac{G(\overrightarrow{X},\overrightarrow{v}(\theta_i))}{\omega_i}.
$$
Now, as polynomials in $X$, we have
$G(\overrightarrow{X},\overrightarrow{v}(X)) =P(X)$, because they
have degree $5$ and agree on the six $\theta_i$. Explicit formulae
for $\Theta$ are
$$
\begin{array}{lll}
p_0 = X_1 + f_1 X_4          & p_2 = X_3 + 2 f_4 X_5 + 2 f_3 X_4 +
f_5 X_6
                             & p_4 =2 f_5 X_4 + 2 X_5\\
p_1 =X_2+2 f_2 X_4 + f_3 X_5 & p_3 = 2 f_4 X_4 + 2 f_5 X_5 + 2 X_6
                             & p_5 = 2 X_4 .
\end{array}
$$

%%%%%%%%%%%%%%%%%%%%%%%%%%%%%%%%%%%%%%%%%%%%%%%%%%%%%%%%%%%%%%%%%%%%%%%%
%%%%%%%%%%%%%%%%%%%%%%%%%%%%%%%%%%%%%%%%%%%%%%%%%%%%%%%%%%%%%%%%%%%%%%%%
%%%%%%%%%%%%%%%%%%%%%%%%%%%%%%%%%%%%%%%%%%%%%%%%%%%%%%%%%%%%%%%%%%%%%%%%

\begin{proposition}\label{conect}
Denoting by $\kappa^{-1}$ and $\kappa_1^{-1}$ the blow-downs from
$\Ss$, respectively $\Sigma$ to $\Kk$ one has $\kappa_1^{-1} =
\kappa^{-1}\circ \Theta$.
\end{proposition}

\begin{proof}

Pick a point $\xi\in\P^3$ and write the equations of the plane
$\Pi_{\xi}\subset\Gg$ of lines through $\xi$ (see (\ref{passer})).
Take $\Hh$ to be defined as in Lemma \ref{lameme}. As seen,
$\Pi_{\xi}$ is tangent to $\Hh$ iff the intersection consists of
two lines:
$$
\Pi_{\xi}\cap\Hh=L_{\xi}\cup L'_{\xi}.
$$
Computing in terms of $\xi$ the coordinates of
$\p_{\xi}=L_{\xi}\cap L'_{\xi}$, we find homogeneous formulae for
$X_i$ in $\xi_i$ of degree $4$:
$$
\p_{\xi}=(X_1(\xi):\dots :X_6(\xi))=\kappa_1(\xi).
$$
We compare now
$$
\Theta\circ\kappa_1(\xi)=(\hat{p}_0(\xi):\dots :\hat{p}_5(\xi))
:\Kk \dashrightarrow \Ss
$$
with $\kappa(\xi)=(p_0(\xi):\dots :p_5(\xi))$ from Lemma
\ref{mapkappa} and obtain
$$
\hat{p}_i p_ 5 - \hat{p}_5 p_i = \delta_i K \qquad\text{with}\,\,
K \,\,\text{given by}\,\ (\ref{KummerEq}),
$$
for $\delta_i$ a homogeneous polynomial in $\xi$.
\end{proof}

Associated with a quadratic complex~$\Hh : H = 0$ there is a set
of~$6$ mutually apolar linear complexes~$\Ll_k$, such that the
polar of any line in~$\Hh$ with respect to~$\Ll_k$ is in~$\Hh$.
If~$G$ and~$H$ are written in diagonal form, these complexes are
$$
\Ll_k : \zeta_k=0\quad\text{for}\quad k=1,\ldots  6.
$$
The action of the  correspondences~$I_k$ on lines in~$\P^3$
translates in coordinates~$\zeta=(\zeta_1 : \ldots : \zeta_6)$ by
\begin{equation}\label{polarite}
I_k(\zeta_i)=(-1)^{\delta_{ik}}\zeta_i,
\end{equation}
which restricts to $\Sigma$. The Kummer is determined by $\Hh$, so
must be invariant under the transformation $I_k$. Therefore the
set of nodes and tropes is invariant (see \cite{Hu},
Section~$30$).

If $u\in\P^3$, we have $I_k(u)=\h(L_{k,u})$ in our previous
notation. Let~$I_{jk}=I_j\circ I_k$. Since
$$
I_j\circ I_{jk}(u) = I_k(u),
$$
we have $I_k(u)=$ plane of lines in~$\Ll_j$ passing through
$I_{jk}(u)$. So, if~$N$ is a node, each plane $I_k(N)$ passes both
through the nodes $N$ and $I_{jk}(N)$ for $j\neq k$, so  is a
trope.

Now let $W_i$ be as in the proof of Lemma \ref{structure}.
% and
%consider the maps
%\begin{equation}\label{autom}
%\kappa_1^*\circ W%_i\circ \kappa_1^{-1} :\Sigma \longrightarrow
%\Sigma%.
%\end{equation}

\begin{proposition}\label{epsilon}
For any~$k$, the map~$I_k$ is the unique automophism of $\Sigma$
such that the following diagram is commutative:
\begin{equation}\label{Iepsi}
\begin{array}{ccc}
\Sigma                 & \harr{I_k}{} &  \Sigma  \\
\varr{\kappa_1^{-1}}{} &              & \varr{}{{\kappa_1^*}^{-1}} \\
\Kk\,                  & \harr{}{W_k} & \Kk^* .
\end{array}
\end{equation}
\end{proposition}

\begin{proof} Let~$\xi\in\Kk$ be a simple point and denote
$\p_\xi=\kappa_1(\xi)$. For a subset $V \subset \Gg$, put
$$
I_k(V)=\{ I_k(l) \in \Gg \mid l \in V \} .
$$
The pencils $I_k(L_\xi)$ and $I_k(L'_\xi)$ are both contained in
the polar plane of $\xi$ with respect to $\Ll_k$, which by Lemma
\ref{polar} is $W_k(\xi)$.  The plane in $\P^5$ parametrizing
lines in $W_k(\xi)$ is therefore tangent to $\Hh$ at
$I_k(L_\xi)\cap I_k(L'_\xi)=I_k(L_\xi\cap
L'_\xi)=I_k(\p_\xi)=I_k\circ\kappa_1(\xi)$. By definition of
$\kappa_1^*$ we have $\kappa_1^*(W_k(\xi))=I_k\circ\kappa_1(\xi)$.
\end{proof}

%%%%%%%%%%%%%%%%%%%%%%%%%%%%%%%%%%%%%%%%%%%%%%%%%%%%%%%%%%%%%%%%%%%%%%%%
%%%%%%%%%%%%%%%%%%%%%%%%%%%%%%%%%%%%%%%%%%%%%%%%%%%%%%%%%%%%%%%%%%%%%%%%
%%%%%%%%%%%%%%%%%%%%%%%%%%%%%%%%%%%%%%%%%%%%%%%%%%%%%%%%%%%%%%%%%%%%%%%%
The following corollary illustrates how the projective duality
(over $k(\theta_k)$) between $\Kk$ and $\Kk^*$ lifts to $\Ss$.
\begin{corollary}\label{epsilon2}
For any~$k$, the map~$\varepsilon^{(k)}$ is the unique automophism
of $\Ss$ such that the following diagram is commutative:
\begin{equation}\label{epsi}
\begin{array}{ccc}
\Ss               & \harr{\varepsilon^{(k)}}{} &  \Ss  \\
\varr{\kappa^{-1}}{}        &                & \varr{}{{\kappa^*}^{-1}} \\
\Kk\,  & \harr{}{W_k}   & \Kk^* .
\end{array}
\end{equation}
\end{corollary}
\begin{proof}

Let $\xi\in\Kk$ and  $\eta\in\Kk^*$ be dual. We have :
\begin{equation}\label{19}
\Theta \circ \kappa_1^*(\eta) \, \mathop{=}^{\ref{duality}} \,
\Theta \circ \kappa_1(\xi) \, \mathop{=}^{\ref{conect}} \,
\kappa(\xi) \, \mathop{=}^{\ref{theoB}} \, \kappa^*(\eta) .
\end{equation}
Note that $\Theta \circ I_k \circ \Theta^{-1} =
\varepsilon^{(k)}$, by (\ref{THETA}), \,(\ref{polarite}) and
(\ref{piji}). Therefore
$$
{\kappa^*}^{-1}\circ\varepsilon^{(k)} \, \mathop{=}^{(\ref{19})}
\, {\kappa_1^*}^{-1} \circ \Theta^{-1} \circ \Theta \circ I_k
\circ \Theta^{-1} \, \mathop{=}^{\ref{epsilon}} \, W_k \circ
\kappa_1^{-1} \circ \Theta^{-1} \, \mathop{=}^{\ref{conect}} \,
W_k \circ \kappa^{-1}.
$$
This is summarized in the following diagram
$$
\begin{matrix}
\Ss & \displaystyle
      \mathop{\longleftarrow}^{\Theta}
                     & \Sigma  & \displaystyle
                                 \mathop{\longrightarrow}^{I_k}
                                                   & \Sigma &
                                                     \displaystyle
                                                     \mathop{\longrightarrow}^{\Theta}
                                                            & \Ss \\
    &  \searrow      & \downarrow
                               &                & \downarrow &
                                                    \swarrow \\
    &                & \Kk     & \displaystyle
                                 \mathop{\longrightarrow}_{W_k}
                                                & K^*
\end{matrix}
$$
\end{proof}

%%%%%%%%%%%%%%%%%%%%%%%%%%%%%%%%%%%%%%%%%%%%%%%%%%%%%%%%%%%%%%%%%%%%%%%%%%%%%
%%%%%%%%%%%%%%%%%%%%%%%%%%%%%%%%%%%%%%%%%%%%%%%%%%%%%%%%%%%%%%%%%%%%%%%%%%%%%
%%%%%%%%%%%%%%%%%%%%%%%%%%%%%%%%%%%%%%%%%%%%%%%%%%%%%%%%%%%%%%%%%%%%%%%%%%%%%

Now Corollary~\ref{epsilon2} is useful for finding explicit
formulae for~$\kappa^*$, simply because
$$
\kappa^*=\kappa^*\circ W_i\circ\kappa^{-1}\circ\kappa\circ
W_i^{-1}=\varepsilon^{(i)}\circ\kappa\circ W_i^{-1}
$$
on an open dense set in~$\Kk^*$. The resulting formulae are huge
and not listed in this paper, since on a given example it is much
easier  to apply successively each map involved.

We now look for the formulae for~$W_i$. Since we need only one
transformation~$W_i$ to find~$\kappa^*$, we may suppose that {\it
all} the roots of ~$F$ are non-zero, for else we use the formulae
given in~\cite{CF}, Chapter~$4$.
%%%%%%%%%%%%%%%%%%%%%%%%%%%%%%%%%%%%%%%%%%%%%%%%%%%%%%%%%%%%%%%%%%%%%%%%
%%%%%%%%%%%%%%%%%%%%%%%%%%%%%%%%%%%%%%%%%%%%%%%%%%%%%%%%%%%%%%%%%%%%%%%%
%%%%%%%%%%%%%%%%%%%%%%%%%%%%%%%%%%%%%%%%%%%%%%%%%%%%%%%%%%%%%%%%%%%%%%%%

\begin{lemma}\label{Wi}
Let~$\theta_i$ be a root of~$F$ and recall that $f_6=1$. Then the
transformation \newline $W_i : \Kk \longrightarrow \Kk^*$
corresponding to the addition of the Weierstrass point
$(\theta_i,0)$ has the following antisymmetric matrix:
$$
A_i=\left(\begin{array}{clcr} 0 & -f_1-2\frac{f_0}{\theta_i} &
a_{13} & \theta_i^2 \cr f_1+2\frac{f_0}{\theta_i} & 0 &
\theta_i^2(f_5+2\theta_i) & -\theta_i \cr -a_{13} &
-\theta_i^2(f_5+2\theta_i) & 0 & 1 \cr -\theta_i^2 & \theta_i & -1
& 0
\end{array}\right)
$$
where~$a_{13}=\theta_i(f_3+2f_4\theta_i+2f_5\theta_i^2+2\theta_i^3)$.
\end{lemma}
\noindent {\sl Note.} \,Each time the vector of coordinates of a
point in $\P^3$ is involved in matrix or scalar product
computations, we view it as a column vector.

\begin{proof}
Suppose we want the matrix corresponding to~$W_1$. Recall
from~\cite{CF}, Chapter~$3$ what the nodes~$N_{ij}$ are :
$$
N_{ij}=\left(1:\theta_i+\theta_j:\theta_i\theta_j:\beta_0(i,j)\right),
$$
where
$$
\beta_0(i,j)=-\prod_{m\neq
i,j}\theta_m-\theta_i\theta_j(\theta_i\theta_j+\sum_{s\neq
t}\theta_s\theta_t)
$$
and in the last sum~$s,t\notin \{i,j\}$.

On writing that~$W_1(N_0)=T_1=(\theta_1^2:-\theta_1:1:0)$ one
finds: $a_{14}=\theta_1^2$, $a_{24}=-\theta_1$, $a_{34}=1$ and
$a_{44}=0$. Now looking at the last coordinate of the
equality~$W_1(N_{1i})=T_i$ one gets a relation:
$$
a_{41}+a_{42}(\theta_1+\theta_i)+a_{43}\theta_1\theta_i+a_{44}\beta_0(1,i)=0,
$$
with a simple solution~$a_{41}=-\theta_1^2$, $a_{42}=\theta_1$,
$a_{43}=-1$ and $a_{44}=0$. One considers then the
equality~$W_1(N_{ij})=T_{1ij}$, where the trope~$T_{1ij}$ has
coordinates
$$
T_{1ij} =
\left(\begin{array}{clcr}(\theta_1+\theta_i+\theta_j)\theta_k\theta_l\theta_m+\theta_1\theta_i\theta_j(\theta_k+\theta_l+\theta_m):&
    -\theta_1\theta_i\theta_j-\theta_k\theta_l\theta_m : \cr
     \theta_1\theta_i+\theta_1\theta_j+\theta_i\theta_j+\theta_k\theta_l+\theta_k\theta_m+\theta_l\theta_m :&
     1\end{array}\right),
$$
where~$k,l,m$ are the indices from~$1$ to~$6$ different
from~$1,i,j$. This yields
$$
a_{41}+a_{42}(\theta_i+\theta_j)+a_{43}\theta_i\theta_j+a_{44}\beta_0(i,j)=1
$$
projectively. The value of the last expression is
$$
v=-\theta_1^2+\theta_1(\theta_i+\theta_j)-\theta_i\theta_j
$$
and this corresponds projectively to~$1$, so the coordinates we
want to find are those of  the trope~$T_{1ij}$, each multiplied
by~$v$. One finishes the computations using Vi\`ete's formulae and
antisymmetry.
\end{proof}

%%%%%%%%%%%%%%%%%%%%%%%%%%%%%%%%%%%%%%%%%%%%%%%%%%%%%%%%%%%%%%%%%%%%%%%%
%%%%%%%%%%%%%%%%%%%%%%%%%%%%%%%%%%%%%%%%%%%%%%%%%%%%%%%%%%%%%%%%%%%%%%%%
%%%%%%%%%%%%%%%%%%%%%%%%%%%%%%%%%%%%%%%%%%%%%%%%%%%%%%%%%%%%%%%%%%%%%%%%

\begin{lemma}\label{polar}
For any point~$\xi\in\P^3$ the plane with dual
coordinates~$W_i(\xi)$ is the polar plane of~$\xi$ with respect to
$\Ll_i$.
\end{lemma}
\begin{proof}
To make a choice, put~$i=1$. Put~$W_1(\xi)=w=(w_1:\ldots:w_4)$. We
have
$$
\sum w_i\xi_i= \xi^T w= \xi^T A_1 \xi = 0 \quad \text{since } A_1
\; \text{is antisymmetric},
$$
so the plane with dual coordinates~$w$ passes through~$\xi$.

Now, a line with Grassmann coordinates~$(X_1:\ldots :X_6)$ passes
through a point~$\xi=(\xi_1:\ldots :\xi_4)\in\P^3$ iff the
following relations hold :
\begin{equation}\label{passer}
\left\{ \begin{array}{ll} \xi_1X_6-\xi_2X_5+\xi_3X_4=0 \\
                          \xi_1X_2+\xi_2X_3-\xi_4X_4=0 \\
                          \xi_1X_1-\xi_3X_3+\xi_4X_5=0.
\end{array}\right.
\end{equation}
Similarly, such a line lies in the plane
$$
\Pi :\quad \sum_{i=1}^4 a_i\xi_i=0
$$
in~$\P^3$ iff
\begin{equation}\label{incluse}
\left\{ \begin{array}{ll} a_2X_4+a_3X_5+a_4X_3=0 \\
                          a_1X_4-a_3X_6+a_4X_2=0 \\
                          a_1X_5+a_2X_6-a_4X_1=0.
\end{array}\right.
\end{equation}
Note that the entries in the matrix~$A_1$ corresponding to~$W_1$
are :
$$
a_{12}=v_1,\quad a_{13}=v_2,\quad a_{14}=v_6,\quad
a_{23}=v_3,\quad a_{34}=v_4,\quad a_{42}=v_5,
$$
where the~$v_i$ are those defined in~\cite{CF}, formula ($17.4.3$)
and~$\theta=\theta_1$. The dual coordinates of the
plane~$W_1(\xi)$ are
$$
\left(\begin{array}{clcr} w_1\cr w_2\cr w_3\cr
w_4\end{array}\right)= A_1 \left(\begin{array}{clcr} \xi_1\cr
\xi_2\cr \xi_3\cr \xi_4\end{array}\right)=
\left(\begin{array}{clcr} v_1\xi_2+v_2\xi_3+v_6\xi_4\cr
-v_1\xi_1+v_3\xi_3-v_5\xi_4\cr -v_2\xi_1-v_3\xi_2+v_4\xi_4\cr
-v_6\xi_1+v_5\xi_2-v_4\xi_3\end{array}\right).
$$

Now, considering relation~(\ref{incluse}) with~$a_i=w_i$, the
conditions that a line passing through $\xi$ with Grassmann
coordinates~$(X_1:\ldots :X_6)$ lie in the plane~$W_1(\xi)$ all
reduce to :
$$
v_1X_4+v_4X_1+v_2X_5+v_5X_2+v_3X_6+v_6X_3=0
$$
(one has to take into account also~(\ref{passer})). But, up to a
constant factor this is the value of~$\zeta_1$ under the change of
variables which brings~$G$ and~$H$ from Lemma~\ref{Kummer2} to
diagonal form (cf. (\ref{XZE}) or~\cite{CF}, formula ($17.4.5$)).
Therefore, a line that passes through~$\xi$  is contained
in~$W_1(\xi)$ iff it belongs to~$\Ll_1$ defined by $\zeta_1=0$.
This is exactly what we want.
\end{proof}
\section{Twist of the desingularized Kummer}\label{twist}

We denote by $\SSS=\Jj(k)$ the Mordell-Weil group of the Jacobian.
There is a map defined by Cassels (\cite{Cas}):
$$
\Phi : \SSS \longrightarrow \Ll = L^* / k^*(L^*)^2 \quad
\text{where} \quad L=k[T]/(F(T)) .
$$
The fake Selmer group is defined by Poonen and Schaefer in
\cite{PS}:
$$
\Sel^{(2)}_{\text{fake}} (k,\Jj )= \{ \xi \in \Ll \mid
          \text{res}_\nu (\xi) \in \Phi_\nu (\Jj (k_\nu)), \text{ for all } \nu \in \Omega \},
$$
where~$\Omega$ is the set of places of~$k$, $\Ll_\nu = L^*_\nu /
k^*_\nu (L^*_\nu)^2$, $L_\nu = L \otimes_k k_\nu$ and $\Phi_\nu :
\Jj (k_\nu) \longrightarrow \Ll_\nu$. The restriction map
$\text{res}_\nu : \Ll \longrightarrow \Ll_\nu$ is induced by~$k
\hookrightarrow k_\nu$.

We present now an idea of M. Stoll and N. Bruin. Let~$\xi \in
\Sel^{(2)}_{\text{fake}} (k,\Jj )$; one looks for an element
$$
[D] = \{ (x,y) , (u,v) \} \in \SSS
$$
such that~$\Phi([D])=\xi$.
%By \cite{Lo} lemma $5.1$ and section $6$, we can represent~$\xi$ by an element of~$\Ll$.
Denote by~$\xi(X)$ a representative polynomial of degree~$5$ of
$\xi$.

The formula~$\Phi([D])=\xi$ can now be written as:
$$
n \, (x-X)(u-X) \equiv \xi(X) \, P(X)^2  \mod F(X)  ,
$$
where~$P(X)$ is the polynomial of degree~$5$ to be found and~$n
\in k^*$.

Define $\Ss^\xi \subset \P^5$ as the projective locus of
polynomials $P(X)$ of degree $5$ such that $\xi(X) P(X)^2 \equiv
\text{quadratic} \mod F(X)$. The surface $\Ss^\xi$ is given by the
three quadratic forms in the coefficients of~$P(X)$:
$$
S^\xi \, : \qquad C^\xi_5 = C^\xi_4 = C^\xi_3 = 0,
$$
where
$$
C^\xi (X) = C^\xi_5 X^5 + C^\xi_4 X^4 + C^\xi_3 X^3 + C^\xi_2 X^2
+ C^\xi_1 X + C^\xi_0
$$
is the polynomial such that
\begin{equation}\label{CXi}
 C^\xi (X) \equiv \xi(X) \, P(X)^2
\mod F(X)  .
\end{equation}
To a rational point on $\Ss^\xi$ corresponds a quadratic rational
polynomial which is a candidate for being of the form $\Phi([D])$
for $[D] \in \SSS$.

If we take~$\xi=1$, we obtain the desingularized Kummer~$\Ss$.

One can interpret this construction as giving a twist of $\Ss$ in
the following way. If $\beta(X) \in \kbar[X]$ is a polynomial of
degree $5$ such that
$$
\beta(X)^2 \equiv \xi(X) \mod F(X) ,
$$
the twist is given by the isomorphism $\alpha :\Ss^\xi
\longrightarrow \Ss$, where
$$
P(X) \longmapsto \alpha(P(X)) \equiv \beta(X) P(X) \mod F(X) .
$$

The twist $\Ss^\xi$ can be diagonalized like $\Ss$. We keep the
notations $P_j(X)$, $\omega_j$ and $\pi_j$ from Section~$3$.
Putting
$$
\xi_j = \xi (\theta_j) ,
$$
we also have
$$
\xi(X)= \sum_{j=1}^6 \frac{\xi_j}{\omega_j} P_j(X) .
$$
Taking into account that
\begin{equation}\label{anulaPiPj}
\begin{aligned}
P_j(X)^2         & \equiv \omega_j P_j(X) \mod F(X) , \\
P_i(X) \, P_j(X) & \equiv 0    \mod F(X) \quad \text{for } i \neq
j ,
\end{aligned}
\end{equation}
finally gives
$$
\xi(X) \, P(X) ^2 \equiv \sum_{j=1}^6 \xi_j\omega_j \pi_j^2 P_j(X)
\mod F(X) .
$$
Since
$$
\begin{aligned}
P_j(X) & = F(X) / \left(  f_6 (X - \theta_j \right) \\
       & = X^5 + (\theta_j + (f_5/f_6) ) X^4
           + (\theta_j^2 + (f_5/f_6) \theta_j + (f_4/f_6)) X^3 +
           \cdots ,
\end{aligned}
$$
the surface $S^\xi$ is obtained in the variables~$\pi_j$ as the
intersection of the three quadrics $S_i^\xi=0$ ($i=0,\, 1 ,\, 2$),
where
$$
%\begin{array}{llcll}
\begin{matrix}
S_0^\xi  & = & C_5^\xi & = & \sum_j \xi_j\omega_j \pi_j^2 , \\
S_1^\xi  & = & f_6 C_4^\xi - f_5 C_5^\xi
                       & = & f_6 \sum_j \theta_j \xi_j\omega_j \pi_j^2 , \\
S_2^\xi  & = & f_6^2 C_3^\xi - f_5 f_6 C_4^\xi
                       + (f_5^2 - f_4 f_6) C_5^\xi
                       & = & f_6^2 \sum_j \theta_j^2 \xi_j\omega_j \pi_j^2 .
%\end{array}
\end{matrix}
$$
Note that if~$\xi=1$ then~$\xi_j=1$ for all~$j$.

\section{Linear automorphisms of $\Ss$}

Keeping Notation \ref{Delta}, we let
\begin{equation}\label{interlines}
\begin{aligned}
\p_i     & = \Delta_0 \cap \Delta_i ,\\
\p_{ij}  & = \Delta_i \cap \Delta_{ij}=\varepsilon^{(i)}(\p_j),\\
\p_{ijk} & = \Delta_{ij} \cap \Delta_{ijk} =
\varepsilon^{(i)}(\p_{jk}).
\end{aligned}
\end{equation}
\begin{remark}\label{griffiths}
Since there are no other lines on $\Ss$ (see \cite{GH}, page
$775$), this is the whole structure of line intersections on
$\Ss$.
\end{remark}
Let $\GL (\Ss)$ be the group of linear
automorphisms of $\Ss$.
%%%%%%%%%%%%%%%%%%%%%%%%%%%%%%%%%%%%%%%%%%%%%%%%%%%%%%%%%%%%%%%%%%%%%%%%%%%%%
\begin{lemma}\label{unique}
Let $A,B \in \GL(\Ss)$ such that
$A_{\mid_{\Delta_0}}=B_{\mid_{\Delta_0}}$. Then $A=B$.
\end{lemma}
\begin{proof}
Let $I \in \GL(\Ss)$ be the identity. If $A \in \GL(\Ss)$ and
$A_{\mid_{\Delta_0}}=I_{\mid_{\Delta_0}}$, then $A$ fixes the
$\p_i$, so invaries the $\Delta_i$. But then $A$ invaries also
$\Delta_{ij}$, the unique line other than $\Delta_0$ which meets
$\Delta_i$ and $\Delta_j$, so $A$ fixes $\p_{ij}$, $j=1,\ldots 6$.
Hence $A_{\mid_{\Delta_i}}=I_{\mid_{\Delta_i}}$. Similarly, one
sees that $A$ is the identity on any of the $32$ lines on $\Ss$,
so $A=I$.
\end{proof}
%%%%%%%%%%%%%%%%%%%%%%%%%%%%%%%%%%%%%%%%%%%%%%%%%%%%%%%%%%%%%%%%%%%%%%%%%%%%%

Let $A \in \GL(\Ss)$. Since $A(\Delta_0)$ is a line, by Remark
\ref{griffiths} there exists a unique involution $\varepsilon \in
\Inv (\Ss)$ such that $\varepsilon \circ A(\Delta_0)=\Delta_0$. We
associate to $A$ the permutation $\sigma \in S_6$ such that
\begin{equation}\label{permutation}
\varepsilon \circ A (\p_i)=\p_{\sigma(i)} \quad\text{for}
\,i=1,\ldots 6.
\end{equation}
Note that $\sigma =\text{id}$ iff $\varepsilon \circ
A_{\mid_{\Delta_0}} = I_{\mid_{\Delta_0}}$ iff $\varepsilon \circ
A = I$ (by Lemma \ref{unique}) iff $A \in \Inv(\Ss)$.

\begin{definition}
$\GL_0(\Ss)$ is the subgroup of $\GL(\Ss)$ of linear automorphisms
$A$ such that $A(\Delta_0)=\Delta_0$.
\end{definition}

%%%%%%%%%%%%%%%%%%%%%%%%%%%%%%%%%%%%%%%%%%%%%%%%%%%%%%%%%%%%%%%%%%%%%%%%%%%%%
%%%%%%%%%%%%%%%%%%%%%%%%%%%%%%%%%%%%%%%%%%%%%%%%%%%%%%%%%%%%%%%%%%%%%%%%%%%%%
\begin{lemma}\label{lemmapermutation}
Let $A \in \GL(\Ss)$ and $\sigma \in S_6$ be the permutation
associated to $A$ by (\ref{permutation}). Then, for any $1\le i
\le 6$ we have:
\begin{equation}\label{eqpermutation}
A \circ \varepsilon^{(i)} = \varepsilon^{(\sigma(i))} \circ A .
\end{equation}
\end{lemma}
\begin{proof}
Let $B=\varepsilon \circ A$. Then \,$B(\Delta_0)=\Delta_0$ and
\,$B(\p_i)=\p_{\sigma(i)}$, so $B(\Delta_i)=\Delta_{\sigma(i)}$.
The unique line cutting $\Delta_{\sigma(i)}$ and
$\Delta_{\sigma(j)}$ is $\Delta_{\sigma(i)\sigma(j)}$ hence,
$B(\Delta_{ij})=\Delta_{\sigma(i)\sigma(j)}$. Then
$$
B(\p_{ij})=B(\Delta_i\cap\Delta_{ij})=B(\Delta_i)\cap
B(\Delta_{ij})=\Delta_{\sigma(i)}\cap
\Delta_{\sigma(i)\sigma(j)}=\p_{\sigma(i)\sigma(j)}.
$$
Now one sees that $(\varepsilon \circ A)^{-1} \circ
\varepsilon^{(\sigma(i))} \circ (\varepsilon \circ A)$ acts like
$\varepsilon^{(i)}$ on $\p_{j}$. By Lemma \ref{unique} and knowing
that $\Inv (\Ss)$ is commutative, we conclude $A \circ
\varepsilon^{(i)} = \varepsilon^{(\sigma(i))} \circ A$.
\end{proof}
%%%%%%%%%%%%%%%%%%%%%%%%%%%%%%%%%%%%%%%%%%%%%%%%%%%%%%%%%%%%%%%%%%%%%%%%%%%%%
%%%%%%%%%%%%%%%%%%%%%%%%%%%%%%%%%%%%%%%%%%%%%%%%%%%%%%%%%%%%%%%%%%%%%%%%%%%%%
\begin{proposition}\label{exacte}
Let $\psi :\GL(\Ss) \longrightarrow \GL_0(\Ss)$ be the map
$A\mapsto \varepsilon\circ A$ defined by formula
(\ref{permutation}). We have an exact sequence of groups
$$
1 \longrightarrow \Inv(\Ss) \longrightarrow \GL(\Ss)
\mathop{\longrightarrow}^{\psi} \GL_0(\Ss) \longrightarrow 1 .
$$
\end{proposition}

%%%%%%%%%%%%%%%%%%%%%%%%%%%%%%%%%%%%%%%%%%%%%%%%%%%%%%%%%%%%%%%%%%%%%%%%%%%%%
Proposition \ref{exacte} implies that $\Inv(\Ss)$ is a normal
subgroup of $\GL(\Ss)$, being the kernel of~$\psi$.

%%%%%%%%%%%%%%%%%%%%%%%%%%%%%%%%%%%%%%%%%%%%%%%%%%%%%%%%%%%%%%%%%%%%%%%%%%%%%
\begin{corollary}
For any linear automorphism $A$ of $\Ss$ not in $\Inv (\Ss)$, the
centralizer of $A$ in $\Inv(\Ss)$ is not equal to $\Inv(\Ss)$.
\end{corollary}

%%%%%%%%%%%%%%%%%%%%%%%%%%%%%%%%%%%%%%%%%%%%%%%%%%%%%%%%%%%%%%%%%%%%%%%%%%%%%
We now show that $\GL_0(\Ss)$ is in bijection with the group of
linear automorphisms of $\Delta_0$ which invary  the set $\{p_i,
\,i=1,\ldots 6\} $.

%%%%%%%%%%%%%%%%%%%%%%%%%%%%%%%%%%%%%%%%%%%%%%%%%%%%%%%%%%%%%%%%%%%
\begin{proposition}\label{prolonge}
Let $\sigma \in S_6$ and $B:\Delta_0 \longrightarrow \Delta_0$ a
linear automorphism of $\Delta_0$ such that for $1 \le i \le 6$,
we have $B(\p_i)=\p_{\sigma(i)}$. Then there exists a unique $A
\in \GL_0(\Ss)$ such that $A_{\mid_{\Delta_0}}=B$.
\end{proposition}
\begin{proof}
Suppose $\sigma$ and $B$ given. If $A$ exists, it is unique by
Lemma \ref{unique} and $\sigma$ is the  permutation associated to
$A$ defined by (\ref{permutation}). Let $\widetilde{A}$ the linear
operator of $\Pp_5$ (polynomials of degree $\leq 5$) associated to
$A$. Let $a,b,c,d \in \kbar$ such that
$$
\widetilde{A}(1)=aX+b \quad \text{and} \quad \widetilde{A}(X)=cX+d
\, .
$$
After some linear algebra and using (\ref{eqpermutation}), we find
that the image of a point $\p \in \Ss$ represented by
$$
P(X)= \sum_j \pi_j P_j(X) \, ,
$$
%we obtain
is
\begin{equation}\label{formuleGL0}
\widetilde{A}(P(X)) = \sum_j  \underbrace
                                     {\left( \pi_j
                                     \frac{\omega_j}
                                          {\omega_{\sigma(j)}}
                                     (a \theta_{\sigma(j)}+b) \right)
                                     }_{\pi'_{\sigma(j)}}
                               P_{\sigma(j)}(X) \, .
\end{equation}
We have to prove that the point $(\pi'_1:\ldots:\pi'_6)$ satisfies
the equations (\ref{p20}).
%defining the surface $\Ss$, that is
%$$
%\sum_j k_j \omega_j {\pi'_j}^2 = 0 \, ,
%$$
%for $k_j=1, \, \theta_j,\, \theta_j^2$ successively.

We show that $k_{\sigma(j)} \omega_{\sigma(j)}
{\pi'}^2_{\sigma(j)} = \alpha_j \omega_j \pi_j^2$ for a quadratic
polynomial $\alpha_j$ in $\theta_j$. We have:
$$
%\begin{aligned}
k_{\sigma(j)} \omega_{\sigma(j)} {\pi'}^2_{\sigma(j)}
%   & = k_{\sigma(j)} \omega_{\sigma(j)}
%                                     \left( \pi_j
%                                     \frac{\omega_j}
%                                          {\omega_{\sigma(j)}}
%                                     (a \theta_{\sigma(j)}+b)
%                                     \right)^2 \\
%   &
      = k_{\sigma(j)} (a \theta_{\sigma(j)}+b)^2
                     \frac{\omega_j}
                          {\omega_{\sigma(j)}}
                     \omega_j \pi_j^2 \, ,
%\end{aligned}
$$
and then
$$
\alpha_j= k_{\sigma(j)} (a \theta_{\sigma(j)}+b)^2
                     \frac{\omega_j}
                          {\omega_{\sigma(j)}} \, .
$$
One can write $\widetilde{A}(X-\theta_i)$ in two ways, using the
fact that $A(\p_i)=\p_{\sigma(i)}$ or linearity of
$\widetilde{A}$:
$$
\mu_j (X-\theta_{\sigma(j)})= \widetilde{A}(X - \theta_j)= cX+d -
\theta_j (aX+b) \, \quad \text{where } \mu_j \in \kbar \, .
$$
Replacing $X=\theta_{\sigma(j)}$, we obtain the formula
\begin{equation}\label{relationtheta}
\theta_j= \frac{c \theta_{\sigma(j)} + d}
               {a \theta_{\sigma(j)} + b} \, ,
\end{equation}
which gives the  relations between the roots of $F(X)$ necessary
for the existence of the linear automorphism $B$.

Now, we calculate $\omega_j$ replacing each $\theta_j$ by the
formula (\ref{relationtheta}):
$$
\begin{aligned}
\omega_j & = \prod_{i\neq j} (\theta_i - \theta_j )
           = \prod_{i\neq j} \left(
                                   \frac{c \theta_{\sigma(i)} + d}
                                        {a \theta_{\sigma(i)} + b} -
                                   \frac{c \theta_{\sigma(j)} + d}
                                        {a \theta_{\sigma(j)} + b}
                             \right) \\
         & = \frac{1}{(a \theta_{\sigma(j)} + b)^4}
             \underbrace{\frac{1}{\prod_i (a \theta_{\sigma(i)} + b)}}_{\text{constant}}
             \prod_{i\neq j} \left(
                                   (\theta_{\sigma(i)} - \theta_{\sigma(j)})
                                   \underbrace{(bc-ad)}_{\text{constant}}
                             \right).
\end{aligned}
$$
Call $\gamma$ the constant part of the equation:
\begin{equation}\label{partial1}
\frac{\omega_j}{\omega_{\sigma(j)}}= \frac{\gamma}{\left( a
\theta_{\sigma(j)} + b \right)^4} \, .
\end{equation}
Replacing (\ref{partial1}) in $\alpha_j$, we have:
$$
\alpha_j= k_{\sigma(j)} (a \theta_{\sigma(j)}+b)^2
          \frac{\gamma}{\left( a \theta_{\sigma(j)} + b \right)^4} = \gamma
          \frac{k_{\sigma(j)}}{\left( a \theta_{\sigma(j)} + b \right)^2} \, .
$$
To see that $\alpha_j$ is a quadratic polynomial in $\theta_j$
(for each $k_j$), we use the formula (\ref{relationtheta}) to
obtain:
$$
\begin{aligned}
a \theta_j - c  & = \frac{ad - bc}{a \theta_{\sigma(j)}+b}
                                                 \\
                & \text{which gives the result for }
                                                 k_{\sigma(j)}=1 ; \\
a^2 \theta_j^2 - c^2 & =
                    \frac{2ac(ad-bc)\theta_{\sigma(j)} + a^2d^2-b^2c^2}
                    {(a \theta_{\sigma(j)}+b)^2} \\
                    & \text{which gives the result for }
                                                 k_{\sigma(j)}=\theta_{\sigma(j)}; \\
b \theta_j - d  & = \frac{(bc-ad) \theta_{\sigma(j)}}{a
\theta_{\sigma(j)}+b} \\
                &  \text{which gives the result for }
                                                 k_{\sigma(j)}=\theta_{\sigma(j)}^2 \, .
\end{aligned}
$$
\end{proof}
%%%%%%%%%%%%%%%%%%%%%%%%%%%%%%%%%%%%%%%%%%%%%%%%%%%%%%%%%%%%%%%%%%%%%%%%%%%%%
%%%%%%%%%%%%%%%%%%%%%%%%%%%%%%%%%%%%%%%%%%%%%%%%%%%%%%%%%%%%%%%%%%%%%%%%%%%%%
%%%%%%%%%%%%%%%%%%%%%%%%%%%%%%%%%%%%%%%%%%%%%%%%%%%%%%%%%%%%%%%%%%%%%%%%%%%%%
%%%%%%%%%%%%%%%%%%%%%%%%%%%%%%%%%%%%%%%%%%%%%%%%%%%%%%%%%%%%%%%%%%%%%%%%%%%%%
%%%%%%%%%%%%%%%%%%%%%%%%%%%%%%%%%%%%%%%%%%%%%%%%%%%%%%%%%%%%%%%%%%%%%%%%%%%%%
%%%%%%%%%%%%%%%%%%%%%%%%%%%%%%%%%%%%%%%%%%%%%%%%%%%%%%%%%%%%%%%%%%%%%%%%%%%%%
%%%%%%%%%%%%%%%%%%%%%%%%%%%%%%%%%%%%%%%%%%%%%%%%%%%%%%%%%%%%%%%%%%%%%%%%%%%%%

%Proposition \ref{prolonge} gives necessary and sufficient conditions
%for the existence of non-commuting involutions of $\Ss$, as we now
%explain.
Proposition \ref{prolonge} gives necessary and sufficient
conditions for the existence of non-trivial elements of
$\GL_0(\Ss)$. For the case of non-commuting involutions of $\Ss$,
we can write this conditions easily.
%%%%%%%%%%%%%%%%%%%%%%%%%%%%%%%%%%%%%%%%%%%%%%%%%%%%%%%%%%%%%%%%%%%%%%%%%%%%%
%%%%%%%%%%%%%%%%%%%%%%%%%%%%%%%%%%%%%%%%%%%%%%%%%%%%%%%%%%%%%%%%%%%%%%%%%%%%%
%%%%%%%%%%%%%%%%%%%%%%%%%%%%%%%%%%%%%%%%%%%%%%%%%%%%%%%%%%%%%%%%%%%%%%%%%%%%%
%%%%%%%%%%%%%%%%%%%%%%%%%%%%%%%%%%%%%%%%%%%%%%%%%%%%%%%%%%%%%%%%%%%%%%%%%%%%%
%%%%%%%%%%%%%%%%%%%%%%%%%%%%%%%%%%%%%%%%%%%%%%%%%%%%%%%%%%%%%%%%%%%%%%%%%%%%%
%%%%%%%%%%%%%%%%%%%%%%%%%%%%%%%%%%%%%%%%%%%%%%%%%%%%%%%%%%%%%%%%%%%%%%%%%%%%%
%%%%%%%%%%%%%%%%%%%%%%%%%%%%%%%%%%%%%%%%%%%%%%%%%%%%%%%%%%%%%%%%%%%%%%%%%%%%%
\begin{remark}\label{translation}
For a curve of genus $2$ defined by
$$
Y^2=F(X) =\prod_{i=1}^6 (X-\theta_i) ,
$$
we may suppose (up to a linear translation) that
$$
\theta_3 + \theta_4 = \theta_5 + \theta_6 \quad \text{or} \quad
\theta_3 \theta_4 = \theta_5 \theta_6 .
$$
\end{remark}
Indeed, suppose that $\theta_3 + \theta_4 \neq \theta_5 +
\theta_6$. On putting $\tilde{\theta}_i=\theta_i + t$ with
$$
t=\frac{\theta_5 \theta_6 - \theta_3 \theta_4}{(\theta_3 +
\theta_4) - (\theta_5 + \theta_6)} .
$$
one gets
$$
\tilde{\theta}_3 \tilde{\theta}_4 = \tilde{\theta}_5
\tilde{\theta}_6 \, .
$$
%%%%%%%%%%%%%%%%%%%%%%%%%%%%%%%%%%%%%%%%%%%%%%%%%%%%%%%%%%%%%%%%%%%%%%%%%%%%%
%%%%%%%%%%%%%%%%%%%%%%%%%%%%%%%%%%%%%%%%%%%%%%%%%%%%%%%%%%%%%%%%%%%%%%%%%%%%%
%%%%%%%%%%%%%%%%%%%%%%%%%%%%%%%%%%%%%%%%%%%%%%%%%%%%%%%%%%%%%%%%%%%%%%%%%%%%%
%%%%%%%%%%%%%%%%%%%%%%%%%%%%%%%%%%%%%%%%%%%%%%%%%%%%%%%%%%%%%%%%%%%%%%%%%%%%%
%%%%%%%%%%%%%%%%%%%%%%%%%%%%%%%%%%%%%%%%%%%%%%%%%%%%%%%%%%%%%%%%%%%%%%%%%%%%%
%%%%%%%%%%%%%%%%%%%%%%%%%%%%%%%%%%%%%%%%%%%%%%%%%%%%%%%%%%%%%%%%%%%%%%%%%%%%%
%%%%%%%%%%%%%%%%%%%%%%%%%%%%%%%%%%%%%%%%%%%%%%%%%%%%%%%%%%%%%%%%%%%%%%%%%%%%%
Note that if $\theta_3 + \theta_4 = \theta_5 + \theta_6$ and
$\theta_3 \theta_4 = \theta_5 \theta_6$ then $\left\{ \theta_3 ,
\theta_4 \right\} = \left\{ \theta_5 , \theta_6 \right\}$, which
is impossible.

%%%%%%%%%%%%%%%%%%%%%%%%%%%%%%%%%%%%%%%%%%%%%%%%%%%%%%%%%%%%%%%%%%%%%%%%%%%%%
%%%%%%%%%%%%%%%%%%%%%%%%%%%%%%%%%%%%%%%%%%%%%%%%%%%%%%%%%%%%%%%%%%%%%%%%%%%%%
%%%%%%%%%%%%%%%%%%%%%%%%%%%%%%%%%%%%%%%%%%%%%%%%%%%%%%%%%%%%%%%%%%%%%%%%%%%%%
%%%%%%%%%%%%%%%%%%%%%%%%%%%%%%%%%%%%%%%%%%%%%%%%%%%%%%%%%%%%%%%%%%%%%%%%%%%%%
%%%%%%%%%%%%%%%%%%%%%%%%%%%%%%%%%%%%%%%%%%%%%%%%%%%%%%%%%%%%%%%%%%%%%%%%%%%%%
%%%%%%%%%%%%%%%%%%%%%%%%%%%%%%%%%%%%%%%%%%%%%%%%%%%%%%%%%%%%%%%%%%%%%%%%%%%%%
%%%%%%%%%%%%%%%%%%%%%%%%%%%%%%%%%%%%%%%%%%%%%%%%%%%%%%%%%%%%%%%%%%%%%%%%%%%%%
\begin{corollary}\label{notcomminv}
Let $A$ be a non-commuting involution of $\Ss$ which fixes
$\Delta_0$. Renumbering the roots of $F$, we have
$$
A(\p_j)=\p_{j+1}
$$
for $j=3,5$. By Remark \ref{translation}, we may suppose that
$\theta_3 + \theta_4 = \theta_5 + \theta_6$ or
$\theta_3\theta_4=\theta_5\theta_6$. Then:

If $A(\p_1)=\p_2$, we have
$$
\theta_1 + \theta_2 = \theta_3 + \theta_4 = \theta_5 + \theta_6
\quad \text{or} \quad
\theta_1\theta_2=\theta_3\theta_4=\theta_5\theta_6.
$$

Otherwise, $\p_1$ and $\p_2$ are fixed by $A$ and then
$\theta_1+\theta_2=0$ and
$$
\theta_1^2=\theta_2^2=\theta_3\theta_4=\theta_5\theta_6 .
$$
\end{corollary}
\section*{Appendix }
{\bf Construction of $\kappa$}

As announced during the proof of Lemma \ref{mapkappa}, we give the
needed polynomials for the construction of $\kappa$. Let $[\XXX] =
[(x,y) + (u,v)]\neq[K_{\Cc}]$, where $yv \neq 0$ and $x \neq u$ be
a divisor class.

There is a unique $M(X)$ of degree $3$ such that
\begin{equation}\label{twiceM}
M(X)^2-F(X)=(X-x)^2(X-u)^2 H(X),
\end{equation}
for a quadratic $H(X)$. There is a unique polynomial $P(X)$ of
degree at most $5$ such that
\begin{equation}\label{polasterisque}
(X-x)(X-u)P(X) \equiv M(X) \text{ mod } F(X).
\end{equation}
Let
\begin{equation}
\begin{array}{lcl}
M(X) & = & \left( m_x (X-x) + 1 \right) \left( (X-u)/(x-u)
\right)^2 y \cr
     &   & + \left( m_u (X-u) + 1 \right) \left( (X-x)/(u-x) \right)^2 v ,
\end{array}
\end{equation}
where~$m_x$ amd~$m_u$ are given by the conditions that the
derivative of~$F(X)-M(X)^2$ vanishes at~$X=x$ and~$X=u$. Hence
\begin{equation}
m_x  =  \displaystyle \frac{F'(x)}{2 F(x)} - \frac{2}{x-u} ,
\qquad m_u  =  \displaystyle \frac{F'(u)}{2 F(u)} - \frac{2}{u-x}
.
\end{equation}
Then consider
\begin{equation}\label{Mstar1}
\begin{array}{rcl}
M^\diamond (X) & = & 2 (x-u)^3 yv M(X) \cr
  & = & (X-u)^2 v \left( \left( F'(x)(x-u)-4F(x) \right) (X-x) + 2 F(x) (x-u) \right) \cr
  & & - (X-x)^2 y \left( \left( F'(u)(u-x)-4F(u) \right) (X-u) + 2 F(u) (u-x) \right) .
\end{array}
\end{equation}
All the terms on the r.h.s are divisible by~$(X-x)(X-u)$,
except~$F(x) (x-u) (X-u)^2$ and~$F(u) (u-x) (X-x)^2$.
But~$F(X)-F(x)=(X-x)F(x,X)$ for some polynomial~$F(x,X)$ and
similarly for~$F(X)-F(u)$. Replacing~$F(x)$ and~$F(u)$
in~(\ref{Mstar1}), one finds that
$$
M^\diamond (X) \equiv (X-x)(X-u) P^\diamond (X) \text{ mod } F(X),
$$
where
$$
\begin{array}{lcl}
P^\diamond (X) &=& \left( F'(x) (x-u) - 4 F(x) - 2 F(x,X)(x-u)
\right) (X-u) v \cr
         & &  - \left( F'(u) (u-x) - 4 F(u) - 2 F(u,X)(u-x) \right) (X-x) y .
\end{array}
$$
The coefficient of~$X^6$ in~$P^\diamond$ is~$-2f_6 (x-u)(y+v)$, so
the polynomial
$$
P^\triangle (X) = P^\diamond (X) +2f_6 (x-u)(y+v) F(X)
$$
is of degree~$5$. Both terms on the r.h.s. are odd and
antisymmetric in~$(x,y)$, $(u,v)$, so multiplying by
$$
\frac{y-v}{(x-u)^2},
$$
we get a polynomial~$P(X)$ whose coefficients are even symmetric
functions of~$(x,y)$, $(u,v)$, so are even functions
of~$k(J(\Cc))$.

We have
$$
\begin{array}{ll}
P^2(X) & \equiv \left( P^\triangle (X)
                 \displaystyle \frac{y-v}{(x-u)^2} \right)^2
                 \equiv \left( P^\diamond (X)
                 \displaystyle \frac{y-v}{(x-u)^2} \right)^2 \\
               & \equiv 4 (x-u)^2 (y-v)^2 y^2 v^2 H(X) \mod F(X) ,
\end{array}
$$
where~$H(X)$ is as in~(\ref{twiceM}). In particular, the
polynomial $P(X)$ satisfying the relation (\ref{polasterisque}) is
the polynomial $P(X)= P^\triangle(X)/\left( 2(x-u)^3 y v \right)$.
One proceeds now as in the proof of Lemma \ref{mapkappa}.
%Finally,
%$$
%{P^\star}^2(X) \equiv \left( P^\triangle (X) \frac{y-v}{(x-u)^2}
%\right)^2
%               \equiv 4 (x-u)^2 (y-v)^2 y^2 v^2 H(X) \text{ mod } F(X) .
%$$
%We get formulae for~$\kappa$ by expressing~$y^2=F(x)$, $v^2=F(u)$,
%$$
%yv=\frac{F_0(x,u)-\beta_0 (x-u)^2}{2}
%$$
%in the coefficients of~$P^\star(X)$, then as the resulting
%coefficients are symmetric functions of $x$ and~$u$, we replace
%them by $\xi_2=x+u$ and~$\xi_3=xu$. Finally we homogenize the
%formulae with respect to~$\xi_1=1$, $\xi_2$, $\xi_3$, $\xi_4$.

%The coefficients of $P^\star(X)$ are now homogeneous polynomials
%of degree $6$ in $\xi$. As  suggested by  formula
%(\ref{polcombination}), we may write $p_i=\alpha_j - \beta_j
%\xi_4^2$. We obtain the following formulae.
\vskip .2truecm

{\bf Polynomial definition of $\kappa$}

\begin{equation}
\begin{array}{ll}
p_0 = & - f_3 f_6 \xi_1 \xi_3^3+ 1/2  f_5^2 \xi_2 \xi_3^3 -2
\xi_3^3 f_4
f_6 \xi_2 + 2 \xi_3^2 \xi_1^2 f_1 f_6-\xi_3^2 \xi_1^2 f_5 f_2 - \\
 & 2 \xi_3^2 \xi_1 \xi_2 f_6 f_2-1/2 \xi_3^2 \xi_1 f_5 \xi_4-1/2
\xi_3^2 \xi_1 \xi_2 f_5 f_3 -\xi_3^2 \xi_2^2 f_6 f_3-2 \xi_3^2
\xi_2
f_6 \xi_4 \\
  & -1/2 \xi_3 f_3 \xi_4 \xi_1^2 -3/2 \xi_3 \xi_1^2 \xi_2
f_5 f_1-\xi_3 \xi_2 f_4 \xi_4 \xi_1-3 \xi_3 \xi_1 \xi_2^2 f_6 f_1 \\
 & -1/2 \xi_3 \xi_2^2 f_5 \xi_4+f_1 f_2 \xi_1^4+\xi_1^3 \xi_2 f_1
f_3+3/2 \xi_1^3 f_1 \xi_4 +\xi_2^2 f_4 f_1 \xi_1^2+\xi_2^3 f_5 f_1
\xi_1 \\
 & +\xi_2^4 f_1 f_6-1/2 \xi_1 \xi_2 \xi_4^2
\\ \\
\notag
\end{array}
\end{equation}
%%%%%%%%%%%%%%%%%%%%%%%%%%%%%%%%%%%%%%%%%%%%%%%%%%%%%%%%%%%%%%%%%%%%%%%
\begin{equation}
\begin{array}{ll}
p_1  = & 2 \xi_1^4 f_2^2-2 \xi_3 \xi_1 \xi_2^2 f_6 f_2+1/2 \xi_1^2
\xi_2^2 f_5 f_1 -1/2 \xi_1^4 f_3 f_1+2 \xi_2^4 f_2 f_6+3 \xi_1^3
\xi_4 f_2 \\
 & +1/2 \xi_3 f_3^2 \xi_1^3 +1/2 \xi_2^3 f_5 \xi_4+\xi_3
\xi_1^2 \xi_2 f_4 f_3-1/2 \xi_3^2 \xi_2^2 f_5^2 +3/2 \xi_3 \xi_1
\xi_2^2 f_5 f_3 \\
& +2 \xi_3^2 f_4 f_6 \xi_2^2-\xi_3 \xi_1^2 \xi_2 f_5 f_2 +\xi_3
\xi_2^2 f_6 \xi_4+2 \xi_3 \xi_2^3 f_6 f_3+\xi_1^2 \xi_4^2+2
\xi_1^3
\xi_2 f_2 f_3 \\
& -\xi_3 \xi_1^2 \xi_2 f_6 f_1+2 \xi_1^2 f_2 f_4 \xi_2^2+3/2
\xi_1^2 \xi_2 f_3 \xi_4 +\xi_1 f_4 \xi_4 \xi_2^2+\xi_1 \xi_2^3 f_6
f_1 \\
& + 2\xi_1 \xi_2^3 f_5 f_2 +2 \xi_3^2 \xi_1^2 f_6 f_2-1/2 \xi_3^2 \xi_1^2 f_5 f_3+\xi_3^2
\xi_4 f_6 \xi_1
\\ \\
\notag
\end{array}
\end{equation}
%%%%%%%%%%%%%%%%%%%%%%%%%%%%%%%%%%%%%%%%%%%%%%%%%%%%%%%%%%%%%%%%%%%%%%%
\begin{equation}
\begin{array}{ll}
p_2 = & 2 \xi_1^2 \xi_2^2 f_4 f_3-f_6 f_5 \xi_1 \xi_3^3+\xi_3^2
\xi_1^2 f_3 f_6 -\xi_3^2 \xi_1^2 f_4 f_5+\xi_3^2 \xi_1 f_5^2 \xi_2
\\
& +\xi_3 f_1 f_6 \xi_1^3 +\xi_3 \xi_1^3 f_3 f_4-2 \xi_3 f_5 f_2
\xi_1^3+2 \xi_3 \xi_1^2 \xi_2 f_4^2 -2 \xi_3 f_5 \xi_4 \xi_1^2
\\
& -\xi_1^4 f_1 f_4+2 \xi_1^4 f_3 f_2 -2 \xi_3 \xi_2 f_6 f_2
\xi_1^2-5 \xi_3 \xi_2^2 f_6 f_3 \xi_1 +2 \xi_3 \xi_1 \xi_2^2 f_5
f_4
\\
& -3 \xi_3 f_6 \xi_4 \xi_2 \xi_1 -3 \xi_3 \xi_2 f_5 f_3 \xi_1^2+2
\xi_2^4 f_3 f_6+\xi_2^3 f_6 \xi_4 +2 f_3 \xi_4 \xi_1^3+2 \xi_2
f_3^2
\xi_1^3 \\
& +2 \xi_2^3 f_5 f_3 \xi_1 -\xi_2 f_5 f_1 \xi_1^3-\xi_2^2 f_6 f_1
\xi_1^2+\xi_2^2 f_5 \xi_4 \xi_1 +\xi_2 f_4 \xi_4 \xi_1^2+2 \xi_3
\xi_2^3 f_6 f_4 \\
& +\xi_3^2 \xi_2^2 f_6 f_5 -4 \xi_3^2 \xi_1 f_4 f_6 \xi_2
\\ \\
\notag
\end{array}
\end{equation}
%%%%%%%%%%%%%%%%%%%%%%%%%%%%%%%%%%%%%%%%%%%%%%%%%%%%%%%%%%%%%%%%%%%%%%%
\begin{equation}
\begin{array}{ll}
p_3 = & -2 f_6^2 \xi_1 \xi_3^3-\xi_3^2 \xi_1^2 f_5^2+2 \xi_3^2
\xi_1^2 f_4 f_6 -\xi_3^2 \xi_2 f_6 f_5 \xi_1+2 \xi_3^2 f_6^2
\xi_2^2
\\
& +\xi_3 \xi_1^3 f_5 f_3 -2 \xi_3 \xi_1^3 f_2 f_6-\xi_3 \xi_1^2
\xi_2 f_6 f_3-2 \xi_3 \xi_1^2 f_6 \xi_4 +2 \xi_3 \xi_1 \xi_2^2
f_5^2
\\
& -4 \xi_3 \xi_1 f_4 f_6 \xi_2^2+2 \xi_3 f_6 f_5 \xi_2^3 -\xi_1^4
f_1 f_5+2 \xi_1^4 f_4 f_2+2 \xi_1^3 f_4 \xi_4+2 \xi_1^3 \xi_2 f_4
f_3 \\
& -\xi_1^3 \xi_2 f_6 f_1+\xi_1^2 \xi_2 f_5 \xi_4+2 \xi_1^2 f_4^2
\xi_2^2+\xi_1 \xi_2^2 f_6 \xi_4 +2 \xi_1 \xi_2^3 f_5 f_4+2 \xi_2^4
f_6 f_4
\\ \\
\notag
\end{array}
\end{equation}
%%%%%%%%%%%%%%%%%%%%%%%%%%%%%%%%%%%%%%%%%%%%%%%%%%%%%%%%%%%%%%%%%%%%%%%
\begin{equation}
\begin{array}{ll}
p_4 = & \xi_3^2 \xi_1^2 f_6 f_5-2 \xi_3^2 \xi_2 f_6^2 \xi_1+\xi_3
f_3 f_6 \xi_1^3 -2 \xi_3 \xi_2 f_5^2 \xi_1^2+2 \xi_3 f_4 f_6 \xi_2
\xi_1^2 \\
& -2 \xi_3 \xi_2^2 f_6 f_5 \xi_1 +2 \xi_3 f_6^2 \xi_2^3-f_6 f_1
\xi_1^4+2 f_5 f_2 \xi_1^4+2 f_5 \xi_4 \xi_1^3 +2 \xi_2 f_5 f_3
\xi_1^3\\
& +2 \xi_2^2 f_5 f_4 \xi_1^2+\xi_2 f_6 \xi_4 \xi_1^2 +2 \xi_2^3
f_5^2 \xi_1+2 \xi_2^4 f_6 f_5
\\ \\
\notag
\end{array}
\end{equation}
%%%%%%%%%%%%%%%%%%%%%%%%%%%%%%%%%%%%%%%%%%%%%%%%%%%%%%%%%%%%%%%%%%%%%%%
\begin{equation}
\begin{array}{ll}
p_5 = & 2 (f_6 \xi_1^2 \xi_3^2-\xi_3 f_5 \xi_2 \xi_1^2 -2 \xi_3
f_6 \xi_2^2 \xi_1+f_2 \xi_1^4+\xi_1^3 \xi_4 +\xi_1^3 f_3 \xi_2+f_4
\xi_2^2 \xi_1^2 \\
& +\xi_2^3 f_5 \xi_1 +f_6 \xi_2^4) f_6 \notag
\end{array}
\end{equation}


\begin{thebibliography}{Cas}


\bibitem[Be]{Be} {\sc A. Beauville},
{\it Surfaces alg\'ebriques complexes}. Soc. Math. France, Paris,
1978.
\bibitem[C]{C} {\sc P. Corn} Tate-Shafarevich Groups and
K3-Surfaces,  to appear in {\it Mathematics of Computation}.

\bibitem[Cas]{Cas} {\sc J.\,W.\,S. Cassels},
The Mordell-Weil group and curves of genus~$2$. {\it Arithmetic
and geometry. Papers dedicated to I.R. Shafarevich, 29--60, Vol
I}, Arithmetic. Birkh\"auser, Boston, Mass, 1983.

\bibitem[CF]{CF} {\sc J.\,W.\,S. Cassels \& E.\,V. Flynn},
{\it Prolegomena to a Middlebrow Arithmetic of Curves of
Genus~$2$}. London Math. Soc. Lecture Note Series 230, Cambridge,
1996.

\bibitem[Fl]{Fl} {\sc E.\,V. Flynn}, The jacobian and formal group
of a curve of genus $2$ over an arbitrary ground field. {\it Mth.
Proc. Cambridge Philos. Soc. 107} (1990), 425-441.

\bibitem[GH]{GH} {\sc Griffiths \& Harris},
{\it Principlies of Algebraic Geometry} John Wiley and Sons,
New-York, 1978.


\bibitem[Hu]{Hu}{\sc R.W.H.T. Hudson}, {\it Kummer's Quartic
Surface}. Cambridge University Press, 1990.

\bibitem[LL]{LL} {\sc A. Logan and R. van Luijk} Nontrivial
elements of Sha explained through K3 surfaces, {\it Mth. Comp. 78}
(2009), 441-483.

\bibitem[PS]{PS} {\sc B. Poonen \& E.\,F. Schaefer},
Explicit descent for Jacobians of cyclic covers of the projective
line, {\it J. reine angew. Math. 488} (1997), 141--188.



%\bibitem[Ser]{Ser} {\sc J.-P. Serre},
%{\it Groupes alg\'ebriques et corps de classes}.
%Hermann, Paris, 1959.


%\bibitem[Si]{Si} {\sc J.H. Silverman},
%{\it The Arithmetic of Elliptic Curves}. Springer-Verlag, New
%York, 1986.



\bibitem[Sto]{St3} {\sc M. Stoll},
Implementing 2-descent for Jacobians of hyperelliptic curves. {\it
Acta Arith. 98} (2001), 245--277.


%\bibitem[FPS]{FPS} {\sc E.\,V. Flynn, B. Poonen \& E.\,F. Schaefer}
%cycles of quadratic polynomials and rational points on a genus~$2$
%curve.
%{\it Duke Math. J. 90} (1997), 435--463.

%\bibitem[Ha]{Ha} {\sc R. Hartshorne},
%{\it Algebraic Geometry}.
%Springer-Verlag, 1977.

%\bibitem[La]{La} {\sc S. Lang},
%{\it Abelian Varieties}.
%Interscience, New York, 1959.

%\bibitem[Lo]{Lo} {\sc V.\,G. Lopez Neumann},
%Interpr\'etation Cohomologique du morphisme $(X-T)$ defini pour les
%jacobiennes des courbes hyperelliptiques
%Preprint

%\bibitem[Ma]{Ma} {\sc Ju.\,I. Manin},
%{\it Cubic Forms}.
%North-Holland, Amsterdam, 1974.

%\bibitem[Mor]{Mor} {\sc L.\,J. Mordell},
%{\it Diophantine Equations}.
%Academic Press, London and New York, 1969.

%\bibitem[Neu]{Neu} {\sc J. Neukirch},
%{\it Class Field Theory}.
%Springer-Verlag, Berlin, 1986.

%\bibitem[Re]{Re} {\sc J. Roberts},
%Chow's Moving Lemma, {\it Algebraic Geometry, Proceedings of the
%5th Nordic Summer-School in Mathematics, 89--96, Oslo 1970},
% Wolters-Noordhoff, Groningen, 1970.

%\bibitem[Sa]{Sa} {\sc P. Satg\'e},
%Morphismes d'une courbe de genre~$2$ vers
%une courbe de genre~$1$. In:
%{\it Arithm\'etique de rev\^etements alg\'ebriques
%(Saint-\'Etienne, 2000)}, 133--146. Soc. Math. France, Paris, 2001.

%\bibitem[Se]{Se} {\sc E.\,S. Selmer},
%Tables for the purely cubic field $K(\sqrt[3]{m})$.
%{\it Avh. Norske Vid. Akad. Oslo I}, 1955, no. 5, 38~pp.

%\bibitem[Sch]{Sch} {\sc E.\,F. Schaefer}
%$2$-descent on the Jacobien of hyperelliptic curves,
%{\it J. Number Th. 51} (1995), 219--232.

%\bibitem[St1]{St1} {\sc M. Stoll},
%Two simple $2$-dimensional abelian varieties
%defined over $\Q$ with Mordell-Weil group of
%rank at least $19$.
%{\it C.R. Acad. Sci. Paris, S\'erie I, 321} (1995), 1341--1344.

%\bibitem[St2]{St2} {\sc M. Stoll},
%An example of a simple $2$-dimensional abelian variety
%defined over $\Q$ with Mordell-Weil group of
%rank at least $20$,
%{\it C.R. Acad. Sci. Paris, S\'erie I, 322} (1996), 849--851.

%\bibitem[Ser2]{Ser2} {\sc J.-P. Serre},
%{\it Cours d'arithm\'etique}.
%PUF $4^{\hbox{e}}$ \'edition, Paris, 1995.

%\bibitem[Wu]{Wu} {\sc C. Wuthrich},
%Une quintique de genre 1 qui contredit le principe de Hasse.
%{\it L'Enseign. Math. 47} (2001), 161--172.
%\bibitem[Be]{Be} {\sc A. Beauville},
%{\it Surfaces alg\'ebriques complexes}.
%Soc. Math. France, Paris, 1978.

%\bibitem[BS]{BS} {\sc B.\,J. Birch \& H.\,P.\,F. Swinnerton-Dyer},
%Notes on elliptic curves,
%I. {\it J. reine angew. Math. 212} (1962), 7--25.

%\bibitem[Ca]{Ca} {\sc J.\,W.\,S. Cassels},
%The rational solutions of the diophantine equation $Y^2=X^3-D$.
%{\it Acta Math. 82} (1950), 243--273.

%\bibitem[Cas]{Cas} {\sc J.\,W.\,S. Cassels},
%The Mordell-Weil group and curves of genus~$2$.
%{\it Arithmetic and geometry. Papers dedicated to I.R. Shafarevich, 29--60,
%Vol I}, Arithmetic. Birkh\"auser, Boston, Mass, 1983.

%\bibitem[CCS]{CCS} {\sc J.-L. Colliot-Th\'el\`ene, D. Coray \& J.-J. Sansuc},
%Descente et principe de Hasse pour certaines vari\'et\'es
%rationnelles. {\it J. Reine. Angew. Math. 320} (1980), 150--191.

%\bibitem[CM]{CM} {\sc D. Coray \& C. Manoil},
%On large Picard groups and the Hasse Principle
%for curves and K3 surfaces.
%{\it Acta Arith. 76} (1996), 165--189.

%\bibitem[Dr]{Dr} {\sc R. Dreier},
%Examples of genus $2$ curves over $\Q$ with Jacobians
%of high Mordell-Weil rank.
%{\it Internat. Math. Res. Notices 18} (1997), 875--880.
\end{thebibliography}
\end{document}